\documentclass[10pt]{amsart}
\usepackage{amscd}
\usepackage{color}
\usepackage[symbol]{footmisc}
\usepackage{amssymb}
\usepackage{amsfonts}
\usepackage{latexsym}
\usepackage{verbatim}
\usepackage{bbm}
\usepackage[shortlabels]{enumitem}
\usepackage[backref=page]{hyperref}
\renewcommand*{\backref}[1]{}
\renewcommand*{\backrefalt}[4]{%
	\ifcase #1 (Not cited).%
	\or        (Cited on page~#2).%
	\else      (Cited on pages~#2).%
	\fi}
\hypersetup{colorlinks=true,linkcolor=blue,citecolor=blue,urlcolor=black}
\newcommand{\N}{\mathbb{N}}
\newcommand{\Z}{\mathbb{Z}}

\newcommand{\R}{\mathbb{R}}
\newcommand{\C}{\mathbb{C}}
\newcommand{\ad}{\operatorname{ad}}
\newcommand{\tr}{\operatorname{tr}}

\newcommand{\g}{\mathfrak{g}}
\newcommand{\h}{\mathfrak{h}}
\newcommand*{\f}[1]{\mathfrak{#1}}
\renewcommand{\epsilon}{\varepsilon}
\renewcommand{\phi}{\varphi}
\renewcommand{\Re}{\mathrm{Re}}
\renewcommand{\Im}{\operatorname{Im}}

\theoremstyle{plain}
\newtheorem{thm}{Theorem}[section]
\newtheorem{prop}[thm]{Proposition}
\newtheorem{lem}[thm]{Lemma}
\newtheorem{cor}[thm]{Corollary}

\theoremstyle{plain}
\newtheorem{maintheorem}{Theorem}

\theoremstyle{definition}
\newtheorem{defn}[thm]{Definition}
\newtheorem{rmk}[thm]{Remark}
\newtheorem{es}[thm]{Example}
\title{A Levi-type decomposition on two-step solvable Lie algebras with a complex structure}

\begin{document}
 
\thanks{This work was partly supported by GNSAGA of INdAM and by the ANR--FAPESP project ANR-21-CE40-0017 {\it Bridges}}

\address{(Elia Fusi) Dipartimento di Matematica G. Peano, Universit\`a di Torino, Via Carlo Alberto 10, 10123, Torino, Italy.}
\email{elia.fusi@unito.it}

\address{(Giovanni Gentili)  Université Paris-Saclay, CNRS, Laboratoire de mathématiques d’Orsay, 91405, Orsay, France.}
\email{giovanni.gentili@universite-paris-saclay.fr}

\author{Elia Fusi and Giovanni Gentili}

\date{\today}

\begin{abstract} 
% We define solvable, semisimple and simple complex structures on  Lie algebras.
We  prove  that a large class of $2$-step solvable Lie algebras equipped with a complex structure $J$ admits a Levi-Malcev type decomposition,  adapted to $J$. As an application, we prove that the Fino--Vezzoni conjecture holds true for $2$-step solvable unimodular Lie algebras. Finally, we give a structural characterisation of $2$-step, unimodular, completely solvable  Lie algebras admitting an SKT metric. 
\end{abstract}

\maketitle

\section{Introduction}

The Levi--Malcev decomposition, see \cite{Levi05, M42}, is a milestone in Lie  theory. Said decomposition asserts that any Lie algebra can be written as a semidirect product of  a semisimple Lie subalgebra, called \emph{Levi subalgebra}, and the maximal  solvable  ideal, called \emph{solvable radical}. Thus, the Levi--Malcev decomposition elevates  semisimple and solvable Lie algebras as building blocks of any  Lie algebra. These two classes have fundamental differences. While semisimple Lie algebras are classified due to the work of  Killing and Cartan, the class of solvable Lie algebras is far more vast and chaotic, resulting into fewer known structural results. On the other hand, whenever a complex structure $J$ is defined on a Lie algebra, it often fails to be compatible with the associated  Levi--Malcev decomposition, i.e. the Levi subalgebra and the solvable radical are usually not $J$-invariant. This naturally raises the question on whether one can find a $J$-adapted version of the Levi--Malcev decomposition on a Lie algebra endowed with a complex structure.  

The main goal of the  paper is to prove a $J$-adapted Levi--Malcev decomposition for a certain class of $2$-step solvable Lie algebras endowed with a complex structure. 

In order to achieve the desired decomposition, we introduce the notions of solvability, semisimplicity and simplicity for a complex structure $J$ on a Lie algebra $\g$. More precisely, we will say that $J$ is \emph{solvable} or that $\g$ is \emph{$J$-solvable} if the following descending series of $J$-invariant subalgebras terminates with $0$:
\[
\g^{[0]}(J):=\g\,, \qquad  \g^{[j+1]}(J):=[\g^{[j]}(J),\g^{[j]}(J)]+J[\g^{[j]}(J),\g^{[j]}(J)]\,,\quad  j \in \N\,.
\]
Using this notion, we can define the \emph{$J$-solvable radical} ${\rm Rad} (\f g, J)$ as the  maximal $J$-solvable ideal of $\f g$. Thus,   we will say that $J$ is \emph{semisimple}, and that $\g$ is \emph{$J$-semisimple}, if ${\rm Rad}(\f g, J)=0$, i.e. $\g$ admits no non-trivial $J$-solvable ideals. Finally, we will say that $J$ is \emph{simple}, and that $\g$ is \emph{$J$-simple}, if $\f g$ is not abelian and it does not have any proper $J$-invariant ideal. We remark that an analogue of $J$-solvability was already introduced and studied in the hypercomplex setting in  \cite{ABB26, Gor23, Gor24}.

With these definitions, it is very natural to wonder how far the analogy with the classical notions of solvability, semisimplicity, and simplicity for a Lie algebra can be carried. While some properties fail, other can still be established within our $J$-adapted framework, although very often there are differences that complicate the scenario. For instance, it is not in general true that a $J$-semisimple Lie algebra is the direct sum of $J$-simple ideals, see Example \ref{ssnonabelian}. However, we prove in Theorem \ref{Semisimple} that such a structural result does in fact hold true when the complex structure is \emph{abelian}. Recall that a complex structure $J$ on $\f g$ is called abelian if and only if 
$$
[JX, JY]=[X,  Y]\,, \qquad X, Y\in\f g\,.
$$
Abelian complex structures occur quite frequently, but only on $2$-step solvable Lie algebras, see \cite{Pet88}. Said complex structures were extensively studied in the literature, see for instance  \cite{ABD11, ABD12, AV16,Bar01,  BD04, FTV22}.
In this setting,   $J$-semisimplicity is equivalent to  the  non-degeneracy of $B^{1,1}$--the $(1,1)$-part of the Killing form--which is reminiscent of Cartan's  criterion for semisimplicity, see Theorem \ref{Semisimple}. 

Among other reasons, this motivates us to focus on $2$-step solvable Lie algebras. Within this setting a simple complex structure is necessarily abelian. We classify all $2$-step solvable $J$-simple Lie algebras. Up to equivalence, the only examples are the Lie algebras of affine motions over the real and complex numbers: $\f{aff}(\R)$ and $\f{aff}(\C)$, see Theorem \ref{affines}. Each of these is equipped with a unique simple and abelian complex structure. As a consequence, all $J$-semisimple Lie algebras with abelian complex structure are direct sums of  copies of $\f{aff}(\R)$ and $\f{aff}(\C)$. In particular they admit a $\Z_2$-module decomposition into two $J$-invariant subspaces which behaves like the Cartan decomposition of semisimple Lie algebras. Furthermore, there exists a holomorphic involution $s$ associated to such a decomposition that plays the role of the Cartan involution and thus, by twisting $B^{1,1}$ with $s$, we obtain a Hermitian metric $B^{1,1}_s$,  see Corollary \ref{corsymmetricdecom}.

With all this settled,  we are able to prove a Levi-type decomposition on any Lie algebra  endowed with an abelian complex structure. 
\begin{maintheorem}\label{mainabelian} Let $\f g$ be a Lie algebra endowed with an abelian complex structure $J$. Then,  there exists a unique $J$-semisimple subalgebra $ \f h $ of $\f g$ such that 
$$
\f g=\f h\ltimes {\rm Rad}(\f g, J)\,.
$$ 
\end{maintheorem}

Both the existence and uniqueness part of the proof of Theorem \ref{mainabelian} are obtained in a step-wise fashion, reducing the problem to simpler situations.  The most substantial part of the existence of the subalgebra $\f h$ is carried out using a holomorphic version of the \emph{Casimir operator}, a fundamental tool in the representation theory of semisimple Lie algebras, see Definition \ref{defn_Cas}.

Apart from giving strong information on the structure of Lie algebras with an abelian complex structures, Theorem \ref{mainabelian} has a pivotal role in the achievement of the desired $J$-adapted Levi--Malcev  decomposition. 
The key observation is that, for $2$-step solvable Lie algebras, the ideal $\g'+J\g'$ always acts holomorphically on $\g'_J:=\g'\cap J\g'$. In particular, the latter is an ideal in the former. Its quotient is then naturally equipped with an abelian complex structure and we can take advantage of our previous discoveries to study this case. Unfortunately, the sought-after decomposition cannot be achieved in full generality, see Example \ref{noLevi}. However, under a suitable additional mild assumption, this expectation can indeed be fulfilled. Such an assumption is formulated in terms of the \emph{mean curvature vector}, defined as follows. Let $\h$ be the $J$-adapted Levi subalgebra of the quotient $(\g'+J\g')/\g'_J$. Then the mean curvature vector $JU$ of $\h$ is defined by the implicit relation
\[
\tr(\ad^{\f h}_X)=B^{1,1}_{s_{\f h}}(JU,X)\,, \qquad X\in \h\,.
\]
The terminology is based on \cite{Einstein}. In general, the mean curvature vector has a fundamental role in the study of the homogeneous Ricci flow and its solitons, see \cite{BL18,FLS24, Heb98, LL14, Lau01}, due to its appearance in the formula for the Ricci curvature of a homogeneous metric, see \cite[Corollary 7.38]{Besse}. We will regard $JU$ as an element of $\g$ by identifying it with any representative of its equivalence class in $\h$, and we will call it the mean curvature vector of $\f g$. Note that, since $2$-step solvability implies that $\g_J'$ is abelian, the action $\ad_{JU}\vert_{\g_J'}$ is well-defined.

With these preparations, we are ready to state our main result, which is the announced $J$-adapted version of the Levi--Malcev  decomposition.

\begin{maintheorem}\label{Main1}
Let $\f g$ be a $2$-step solvable Lie algebra endowed with a complex structure $J$ and suppose that $2\notin \mathrm{Spec}(\ad_{JU}\vert_{\g'_J})$, where $JU$ is the mean curvature vector of $\g $. Then, there exists a unique $J$-semisimple subalgebra $\f h$  of $\f g$ such that
$$
\f g=\f h\ltimes {\rm Rad}(\f g, J)\,. 
$$
\end{maintheorem}

The condition on the spectrum of  $JU$ is equivalent to  the vanishing of certain cohomology,   expressed via the adjoint representation induced by the mean curvature vector acting on $\g'_J$, see Definition \ref{cohomolob} and Lemma \ref{obstruction}. We emphasise that when the complex structure is abelian, such additional assumption is vacuous. Thus, Theorem \ref{Main1} is in fact a generalisation of Theorem \ref{mainabelian}. Although, more involved, the proof of Theorem \ref{Main1} follows a similar strategy to that of Theorem \ref{mainabelian}. In this case,  the most laborious situation is when $\f g$ acts holomorphically on $\f g_J'$. To deal with this, after further reduction to a simpler case, the result is obtained  using an operator that is completely determined by the mean curvature vector $JU$ and the eigenvalues of $\ad_{JU}|_{\f g_J'}$.

\medskip

Our main interest in applications of Theorem \ref{Main1} stems from Hermitian geometry. In particular, the $J$-adapted Levi--Malcev decomposition provides a fruitful structure for identifying obstructions to the existence of certain special Hermitian metrics, as well as incompatibilities among them. For instance, we shall  focus on the so-called Fino--Vezzoni conjecture. We recall here the statement:

\medskip

{\bf Fino--Vezzoni Conjecture \cite{FV15}:} Let $(M,J)$ be a compact complex manifold. If $M$ admits an SKT metric and a balanced metric both compatible with $J$, then $(M,J)$ admits a K\"ahler metric.

\medskip
Recall that a Hermitian metric $g$  on a complex manifold $(M^n, J)$ is called \emph{SKT} if its fundamental form $\omega(\cdot, \cdot):=g(J\cdot, \cdot )$ satisfies $dd^c\omega=0$, where $d^c:=J^{-1}dJ$, see \cite{Bis89}, while $ g $ is called \emph{balanced} if $d\omega^{n-1}=0$, see \cite{G77, M82}.  Many examples of SKT and balanced metrics were constructed on Lie algebras, see e.g. \cite{AV16,BF24, FP22,FP23,FP223,FPS04, FS21, Uga07}

The Fino-Vezzoni conjecture has attracted a lot of interest in the recent literature and there is plenty of partial confirmation, especially in the context of Lie algebras. Most notably, the aforementioned conjecture has been proved on nilpotent Lie algebras in \cite{FV16} and it was addressed and confirmed on certain subclasses of $2$-step solvable Lie algebras in \cite{FS22} and on almost-abelian Lie algebras in \cite{FP23}. In \cite{FS22} the authors also show that in the left-invariant setting, unimodularity is a necessary condition. In \cite{GP23, Joseph26}, the Fino--Vezzoni conjecture was proved to hold on semisimple Lie algebras endowed with \emph{regular} complex structures, while,  in \cite{CZ24}, the conjecture is proved to be true on Lie algebras admitting a codimension 2 abelian ideal. We mention, in particular, the very recent work \cite{FV26}, where the conjecture is established for all compact quotients of a Lie group equipped with a left-invariant complex structure that is Chern--Ricci flat. More validations of the aforementioned conjecture can be found in \cite{C14,CRS22,FGV19, FLY12, FG26}. 

While approaching this problem on solvable Lie algebras, it is quite natural to first restrict to the $2$-step solvable case. We emphasise that this framework is further particularly well motivated by the classical result of  Hano, which states that  left-invariant K\"ahler structures on unimodular Lie algebras can only be supported by $2$-step solvable ones \cite[Theorem II]{H57}.

Interestingly enough, the existence of a compatible SKT metric on a $2$-step solvable Lie algebra automatically guarantees that the hypothesis of Theorem \ref{Main1} is satisfied, see Theorem \ref{SKTLevi}, which thus yields a $J$-adapted Levi--Malcev decomposition. This piece of information enables us to confirm the Fino--Vezzoni conjecture on all $2$-step solvable unimodular Lie algebras.

\begin{maintheorem}\label{mainfv}
Let $\g$ be a $2$-step solvable unimodular Lie algebra endowed with a complex structure $J$. If there exist an SKT metric and a balanced metric both compatible with $J$, then $J$ is solvable. In particular there exists also a K\"ahler metric compatible with $J$.
\end{maintheorem}
The theorem is obtained in two main steps. First, we show that the semisimple part of $\f g$ must vanish, hence we conclude that $\f g$ is $J$-solvable. This step is deduced by means of the characterisation of the existence of balanced metrics in terms of currents,  see \cite{M82}. Next, we show that a unimodular $2$-step solvable Lie algebra is $J$-solvable if and only if it is Chern--Ricci flat, see Proposition \ref{rhocore}. The theorem then follows from \cite{FV26}.
 
The connection between $J$-solvability and Chern--Ricci flatness is  interesting \emph{per se} and it is exclusive of the $2$-step solvable case, see Example \ref{exsimplerigid3step}.  On the other hand, Chern--Ricci flatness together with the existence of an SKT metric and some additional mild assumptions, yield information on the spectrum of the adjoint representation, see Proposition \ref{cazziarmati}.  With this in hand, we can prove the following. 
\begin{maintheorem}\label{mainsktcompletely}
Let $\f g$ be a $2$-step, unimodular,  completely solvable Lie algebra admitting an SKT metric. Then, there exists $p\ge 0$ such that 
$$
\f g=\f{aff}(\R)^p\ltimes \f n_J\,,
$$
where $\f n_J$ is the maximal $J$-invariant nilpotent ideal of $\f g$. 
\end{maintheorem}
The core of the proof of Theorem \ref{mainsktcompletely} consists in showing that a $J$-solvable,  $2$-step, unimodular, completely solvable Lie algebra  admitting an SKT metric is in fact nilpotent, see Theorem \ref{completelysolvSKT1}. The result is then a consequence of Theorem \ref{Main1}.

For higher steps of solvability, the situation appears to be much more complicated, and from time to time we shall exhibit counterexamples to properties that hold in the $2$-step solvable scenario. Nevertheless, we shall collect in Subsection \ref{subsechigher} some interesting results that remain valid in this more general setting.

\medskip

The paper is organised as follows. In Section \ref{secprel}, after recalling some necessary preliminaries, we define and study  basic properties of solvable, semisimple and simple  complex structures. The remainder of the section is devoted to recalling  known facts about special Hermitian metrics and the Chern--Ricci form of a Lie algebra. In Section \ref{secabel}  we first  characterise $J$-semisimple Lie algebras with abelian complex structure as direct sums of $J$-simple Lie algebras, and we give a classification of the latter. We then provide the proof of Theorem \ref{mainabelian}.
In Section \ref{seclevi}, we discuss the  obstruction appearing in Theorem \ref{Main1}, together with some useful consequences. We then present the proof of Theorem \ref{Main1}.  Finally, Section \ref{secsolv} is devoted to the discussion of the proofs of Theorem \ref{mainfv} and Theorem \ref{mainsktcompletely}, as well as to other properties of $J$-solvable Lie algebras. 

\medskip

{\bf Acknowledgements. } The authors would like to express their deep gratitude to James Stanfield for numerous and insightful  conversations and for his interest in the  paper. The authors would  like to thank Asia Mainenti for many inspiring discussions over the preparation of this work. The authors would like also to thank A. Andrada, R. Arroyo, A. Fino, H. Kasuya, R. Lafuente, J. Lauret, and L. Vezzoni for their interest in the present paper.  The first named author would like to express his gratitude to Andrei Moroianu and the Laboratoire de Mathématiques d'Orsay for  the  warm hospitality  and  for the pleasant stay  during which the paper was concluded.

\medskip

{\bf Notations.} We collect here some notations that will be adopted throughout the text.
\begin{itemize}
    \item Let $\g$ be a Lie algebra, then we denote by ${\f z}(\g):=\{ X \in \g \mid \ad_X=0\}$ the centre of $\g$, and by  $\g':=[\g,\g]$ the commutator ideal of $\g$. Finally, $\f n(\f g)$ denotes the nilradical of $\f g$, i.e. the maximal nilpotent ideal of $\f g$. We will simply write  $\f n$ if no confusion is possible. 
    \item  Let $\f s$ be a subalgebra of a Lie algebra $\f g$. We denote by $C_{\f g}(\f s):=\{X\in\f g\,\, |\,\, \ad_X\f s=0\}$ the centraliser of $\f s$ in $\f g$, and by $N_{\f g}(\f s):=\{X \in \f g\, \,|\, \, \ad_X\f s\subseteq \f s\}$ the normaliser of $\f s $ in $\f g$.
    
    \item Let $\f s $ be a subalgebra of a Lie algebra $\g $ equipped with a complex structure $J$. We will use the notation $\f s_J:=\f s \cap J\f s$, in particular $\g'_J:=\g'\cap J\g'$ and $\f n_J:=\f n \cap J \f n$.
    \item Let $\f g_1, \f g_2$ be Lie algebras equipped with a complex structure. If $\g_1$ and $\f g_2$ are holomorphically isomorphic we write $\g_1 \simeq \g_2$.
    \item $B_\g(X,Y):=\tr(\ad_X\ad_Y)$, $X,Y\in \g$, is the Cartan--Killing form of the Lie algebra $\g$. If no confusion is possible we will simply write $B$ in place of $B_\g$.
    \item If $T$ is a tensor, $T^{p,q}$ denotes the $(p,q)$-part of $T$ with respect to the complex structure $J$.
    \item We will write the signature of a  symmetric bilinear form with the notation $(n_+,n_-,n_0)$, where $n_+$ and $n_-$ denote the number of negative and positive eigenvalues respectively, counted with multiplicities, while $n_0$ denotes the nullity.

\end{itemize}

\section{Preliminaries}\label{secprel}
Before discussing Lie algebras with a complex structure, we begin by recalling some standard definitions  and facts concerning  Lie algebras. 
\begin{defn}
Let $\f g$ be a  Lie algebra. We say that $\f g$ is 
\begin{enumerate}
\item \emph{completely solvable} if ${\rm Spec}(\ad_X)\subseteq \R$, for any $X\in \f g$;
\item \emph{of real type} if,    for any $X\in \f g$, $\ad_X$ is either nilpotent or has at least one eigenvalue with non-zero real part; 
\item \emph{of rigid type} if ${\rm Spec}(\ad_X)\subseteq \sqrt{-1}\R$, for any $X\in \f g$. 
\end{enumerate}
\end{defn}

 The following well-known result will be useful in what follows.  Its proof can be found in \cite[Theorem 3.8.3]{Var84}.
\begin{lem}\label{lem_nilideali}
Let $\f g$ be a  Lie algebra and $\f k$ be an ideal of $\f g$. Then $\f n(\f k)=\f n(\f g)\cap \f k.$
\end{lem}

In order to prove Theorem \ref{Main1}, we will need to define the cohomology of a Lie algebra with values in  a given representation. We will just define  the  second  degree cohomology, which will be the most relevant for us. We refer to \cite{Var84} for further details. 
Let $\tau$ be a representation of $\f g $ on a finite-dimensional vector space $V$.  We consider 
$$
V^1(\f g, \tau):=\f g^*\otimes V\,,  \qquad V^2(\f g, \tau):=\Lambda^{2}\f  g^*\otimes V\,,
$$
and, for any $L\in V^1(\f g, \tau)$ and $\alpha \in V^2(\f g, \tau)$,  we define the differential
$$
\begin{aligned}
(d_{\tau}L)(X, Y):=&\, \tau(X) L(Y)-\tau(Y)L(X)-L([X, Y])\,, \quad X,Y\in \f  g\,,\\
(d_{\tau}\alpha)(X, Y, Z):=&\, -\sum_{{\rm cyclic}}\alpha(X, [Y, Z])+ \tau(X)\alpha(Y, Z)\,, \quad X, Y,Z \in\f g\,. 
\end{aligned}
$$
We are now in the position to define the  cohomology of $\f g$ with values in $\tau$.

\begin{defn}\label{cohomologiagen}
Let $\tau$ be a representation of a Lie algebra $\f g$ on a finite-dimensional vector space $V$. Then, the second degree cohomology of $\f g $ with values in $\tau$ is 
 $$
H^2(\f g, \tau):=\frac{\ker d_{\tau}\cap V^2(\f g,\tau)} {{\rm Im}\, d_{\tau}\cap V^2(\f g,\tau)}\,.
 $$
 \end{defn}

Next, let $\f g$ be a Lie algebra endowed with a complex structure $J$. Recall that the integrability of $J$ is encoded by the vanishing of the Nijenhuis tensor:
\[
N_J(X,Y):=[X,Y]+J[JX,Y]+J[X,JY]-[JX,JY] \equiv 0\,, \qquad X,Y\in \g\,. 
\]
Equivalently,%  the integrability of $J$ can be rewritten as 
\begin{equation}\label{integrability}
[J, \ad_{JX}]=J[J, \ad_X]\,, \quad X \in \f g\,.
\end{equation}

We prove here the following basic lemma.
\begin{lem}\label{Lem2stepideal}
Let $\f g$ be a Lie algebra endowed with a complex structure $J$. Then
\begin{equation}\label{behaviourofJ2step}
[X, JY]+[JX, Y]\in \f g_{J}'\,, \qquad X,Y\in \g\,.
\end{equation}
Moreover, if $\g$ is $2$-step solvable, then $\f g_{J}'$ is a $J$-invariant abelian ideal of $\g'+J\g'$  and
\begin{equation}\label{adXg'_J}
J\operatorname{ad}_X=\operatorname{ad}_{JX}\,, \qquad X\in \g'_J\,.
\end{equation}
Furthermore,
$$
[\ad_{JY}, J]|_{\f g_{J}'}=0\,, \qquad Y\in \f g'+J\f g'\,,
$$
i.e. $\f g'+J\f g'$ acts holomorphically on $ \f g_J'$.
\end{lem}
\begin{proof}
The first statement is merely a consequence of the integrability of $J$. Since $\g$ is 2-step solvable, $\f g_J'$ is  clearly abelian. In order to see $\f g_J'$ is an ideal of $\g'+J\g'$,  we observe that, since $\f g $ is $2$-step solvable, $[\f g_J', \f g']=0$. Hence, it is sufficient to prove  $[\f g_J', J \f g']\subseteq \f g_J'$. On the other hand, using again $2$-step solvability and integrability of $J$, for any $Y\in \f g'$ and $X\in\f g_J'$ we have
$
[X, JY]=-J[JX, JY]. 
$
This proves the remaining claims.
\end{proof}

The next example shows that $\f g'_J$ need not be an ideal of $\g'+J\g'$ when the solvability step is greater than $2$.

\begin{es}\label{exsimplerigid3step}
Consider the $3$-step solvable rigid type Lie algebra $\f g=\langle e_1,\dots,e_4 \rangle $ with structure equations:
\[
[e_2,e_3]=-e_1\,, \qquad [e_2,e_4]=-e_3\,, \qquad [e_3,e_4]=e_2\,.
\]
In the notations of \cite{SW14} this Lie algebra is $\f s_{4,7}$. We equip $\g$ with the complex structure
\[
Je_1=e_4\,, \qquad Je_2=e_3\,.
\]
Note that $\g'_J=\langle e_2,e_3 \rangle$ is not an ideal.
\end{es}

\subsection{Types of complex structures on Lie algebras}
In this subsection, we define the types of complex structures mentioned in the introduction and studied throughout the paper. We begin with the notion of abelian complex structures. 

\begin{defn}
Let $\f g $ be a Lie algebra endowed with a complex structure $J$. $J$ is called \emph{abelian} if
$$
[JX, JY]=[X, Y]\,,\quad X, Y\in\f g\,.
$$
\end{defn}
We recall that a Lie algebra equipped with an abelian complex structure is necessarily 2-step solvable \cite{Pet88}, furthermore the centre is $J$-invariant. Finally, $\f g_J'\subseteq \f z(\f g'+J\f g')$. 
Using \eqref{behaviourofJ2step}, we notice  that if $\g'_J= 0$, then  $J$ is abelian. As we will see in what follows, abelian complex structures will play a key role in our study of complex structures on  $2$-step solvable Lie algebras.

\medskip
Let $ \g $ be a Lie algebra equipped with a complex structure $J$. We define recursively the following sequences of $J$-invariant subalgebras: for $j\in \N$, 
\[
\g^{[0]}(J):=\g\,, \qquad  \g^{[j+1]}(J):=[\g^{[j]}(J),\g^{[j]}(J)]+J[\g^{[j]}(J),\g^{[j]}(J)]\,.
\]
We call  $ \{ \g^{[j]}(J) \}_{j\in \N} $  the \emph{$J$-adapted derived series} of the Lie algebra $ \g $. 
As a first difference from the classical derived series, we note  that, in general, $\g^{[j]}(J)$ is only an ideal of $\g^{[j-1]}(J)$ and not of $\f g$, as shown by the next example.

\begin{es}
Let  $\g=\langle e_1,\dots,e_6 \rangle$ be the Lie algebra defined by
\[
[e_1,e_2]=e_1\,, \qquad [e_5,e_2]=e_3\,, \qquad [e_6,e_2]=e_4\,,
\]
with complex structure
\[
Je_1=e_2\,, \qquad Je_3=e_4\,, \qquad Je_5=e_6\,.
\]
In \cite{SW14} this Lie algebra is called $\f s_{6,1}$. Note that $\g^{[1]}(J)=\g'+J\g'=\langle e_1,e_2,e_3,e_4 \rangle$ and $\g^{[2]}(J)=\langle e_1,e_2 \rangle$, which is not an ideal of $\f g$.
\end{es}

 Once the $J$-adapted derived series is defined, we can define solvable complex structures as follows.
\begin{defn}\label{solv}
Let $\g $ be a Lie algebra equipped with a complex structure $J$. We say that $J$ is \emph{solvable} if $ \g^{[j]}(J)=0 $, for some $ j\in \N $. Moreover, we say that it is \emph{$ s $-step solvable}, or that $ s $ is the \emph{solvability step}, if $ \g^{[s]}(J)=0 $ but $ \g^{[s-1]}(J)\neq 0 $.
We will also say that a Lie algebra $\g$ is \emph{$J$-solvable} if it is equipped with a solvable complex structure $J$.
\end{defn}

\begin{rmk}
In \cite{CFGU00} it has been introduced a notion of nilpotency for complex structures. Using the characterisation given in \cite{GZZ23} it is easy to see that a nilpotent complex structure is solvable in the sense of Definition \ref{solv}.  Recall that a complex structure $J$ is bi-invariant if and only if $J[X, Y]=[X, JY]$, for any $X, Y\in \f g$.  It is fairly easy to  observe that bi-invariant complex structures on solvable Lie algebras are always solvable. 
\end{rmk}

Whenever the complex structure $J$ is not solvable, the $J$-adapted derived series stabilises to a non-zero subalgebra $\f s$ such that  $\f s=\f s^{[1]}(J)$.  In analogy with perfect Lie algebras, we define the following class of complex structures.

\begin{defn}
Let $\g$ be a Lie algebra with a complex structure $J$. If
\[
\f g= \f g'+J\f g'\,,
\]
we will say that $\g$ is \emph{$J$-perfect}. Equivalently, $\f g $ is $J$-perfect if and only if $\f g=\f g^{[1]}(J).$
\end{defn} 

Obviously $J$-perfect Lie algebras cannot be $J$-solvable. We also note that such a condition has already been investigated in the literature under different terminology. For instance in \cite{FS22}, $J$-perfect Lie algebras are called of {\em pure type III}.

We now prove some basic properties of $J$-solvability.

\begin{lem}\label{Lem:basicnilsol}
Let $\g$ be a Lie algebra with a complex structure $J$.
\begin{enumerate}
    \item If $ \g $ is nilpotent then $J$ is solvable;
    \item If $J$ is solvable then $\g$ is solvable.
\end{enumerate}
\end{lem}
\begin{proof}
The assertion \emph{(2)} is clear because  the derived series of $\g$ is contained in the $J$-adapted one. Finally, \emph{(1)} follows from the results in \cite{Sal01}, according to which a nilpotent Lie algebra is never $J$-perfect implying that the $J$-adapted derived series must terminate with $0$.
\end{proof}

It is not difficult to show that $J$-invariant subalgebras, as well as images of $J$-solvable Lie algebras under holomorphic Lie algebra homomorphisms, are again $J$-solvable. In particular quotients of $J$-solvable Lie algebras by $J$-invariant ideals are $J$-solvable. Similarly to the case of solvable Lie algebras, the following holds. 
\begin{lem}\label{lemmagrullo}
Let $ \f k\subseteq \g $ be a $J$-solvable ideal such that  $ \g/\f k $ is $J$-solvable, then  $ \g $ is $J$-solvable. Finally, the intersection and sum of $J$-solvable ideals is $J$-solvable.
\end{lem}
\begin{proof}
Denoting with $\pi\colon \f g \to \f g/\f k $ the canonical projection, we have that  $ \pi(\g^{[j]}(J))=(\g/\f k )^{[j]}(J) $ vanishes for  some  $ j\in \N $. Therefore,  $ \g^{[j]}(J)\subseteq \ker\pi=\f k  $ and thus $ \g^{[l+j]}(J)=(\g^{[j]}(J))^{[l]}(J)\subseteq \f k ^{[l]}(J)=0 $,  for  some $ l\in \N $. Finally, if $ \f k_1,\f k_2 $ are $J$-solvable ideals it is obvious that $ \f k_1 \cap \f k_2 $ is $J$-solvable. From the isomorphism $ (\f k_1+\f k_2)/\f k_1\simeq \f k_2/(\f k _1\cap \f k_2) $ we note that $ (\f k_1+\f k_2)/\f k_1 $ is solvable and thus also $ \f k_1+\f k_2 $ is.
\end{proof}

Since $J$-solvable ideals are closed under the sum the following definition is well-given.

\begin{defn}
Let $\g$ be a Lie algebra equipped with a complex structure $J$. Define ${\rm Rad}(\f g, J)$ as the maximal $J$-solvable ideal of $\g$. This ideal will be called the \emph{$J$-solvable radical of} $\f g$. 
\end{defn}
Since we are interested in working on solvable Lie algebras, we will never take into account the usual solvable radical, that is the maximal solvable ideal of a Lie algebra, therefore we shall denote ${\rm Rad}(\f g)$ the $J$-solvable radical of $\f g$ with no ambiguity. 

\begin{defn}
Let $\g $ be a Lie algebra equipped with a complex structure $J$. We say that $J$ is \emph{semisimple} if $\g$ admits no non-zero $J$-solvable ideals, i.e. ${\rm Rad}(\f g)=0$. Furthermore, we will say that $J$ is \emph{simple} if $\g $ is not abelian and admits no  proper $J$-invariant ideals. A Lie algebra $\g$ equipped with a simple or semisimple complex structure will be called \emph{$J$-simple} or \emph{$J$-semisimple}, respectively.
\end{defn}

Let us now establish some basic properties that follow from these definitions.

\begin{lem}\label{lemss}
Let $\g$ be a  Lie algebra with a  complex structure $J$. Then
\begin{enumerate}
\item if $J$ is simple,   it is semisimple and  $\f g$ is $J$-perfect;
\item if $\f g$ is $2$-step solvable  and $J$ is simple, then $J$ is abelian and $\f g=\f g'\rtimes J \f g'$;
\item $ \g / {\rm Rad}(\f g) $ is $J$-semisimple.
\end{enumerate}
\end{lem}
\begin{proof}
To prove \emph{(1)}, we note that if $J$ is simple it is obviously semisimple. Moreover,   $\g'+J\g'$ is a $J$-invariant ideal of $\g$, hence we either have $\g'+J\g'=0$ or $\g'+J\g'=\g$. In the first case we would have that $\g$ is abelian, which is not possible, deducing that  $\g$ must be $J$-perfect.

As for \emph{(2)}, thanks to \emph{(1)}, and Lemma \ref{Lem2stepideal} we know that $\f g_J'$ is an ideal of $\f g$. Now, simplicity forces either $\f g_J'=0$ or $\f g_J'=\f g$. However, the latter cannot happen since $\f g $ is solvable. Therefore, $\f g_J'=0$  implying that $J$ is  abelian and $\f g=\f g'\rtimes J \f g'$.

Finally, we prove \emph{(3)}. To any $J$-solvable ideal $ \f k $ of the quotient $ \g/ {\rm Rad}(\f g) $ corresponds a $J$-solvable ideal  of $ \g $. Therefore, the latter ideal must be contained in $ {\rm Rad}(\f g) $, implying that $ \f k $ is trivial in the quotient.
\end{proof}

\subsection{Special Hermitian metrics.}
In view of applications of  Theorem \ref{Main1} to special Hermitian metrics, we recall here some essential notions regarding them. We start by giving the definition of the two kinds of special metrics we shall be interested in.

\begin{defn}
Let $(M^n, J)$ be a complex manifold. A Hermitian metric $g$ is called \emph{SKT} or \emph{pluriclosed} if and only if $dd^c\omega=0$, where $\omega$ is the associated fundamental form. On the other hand $g$ is called \emph{balanced} if and only if $d\omega^{n-1}=0$.
\end{defn}

In order to address the Fino--Vezzoni conjecture, it is useful to have at our disposal characterisations of the existence of balanced and SKT metrics. In particular, we will use the following criterion for the existence of balanced metrics expressed in terms of currents due to Michelsohn \cite[Theorem A]{M82}.

\begin{thm}\label{Michelsohn1}
Let $(M, J)$ be a compact complex manifold. Then, $(M, J)$ does not admit any balanced metrics if and only if there exists a $d$-exact current with non-zero and positive $(1,1)$-part.
\end{thm}

We will exclusively be concerned in the case in which our manifold $M$ is a quotient of a Lie group by a co-compact lattice endowed with a left-invariant complex structure. In this setting, we can adopt the symmetrisation technique \cite{Bel00} and infer that  the existence of balanced or SKT metrics on $M$ is equivalent to the existence of a left-invariant one \cite{FG04,Uga07}. This allows to reduce the problem of studying such metrics at the Lie algebra level. For this reason, unless otherwise stated, we will henceforth work exclusively on Lie algebras. Within this framework, in Theorem \ref{Michelsohn1} compactness can be replaced with unimodularity. Furthermore, we will only need the following simplified version of one implication. For the readers' convenience we provide a proof.

\begin{cor}\label{Michelsohn}
Let $\g$ be a unimodular Lie algebra equipped with a complex structure $J$. If $\g$ admits a $d$-exact, non-zero, and semi-positive $(1,1)$-form, then it does not carry any balanced metrics.
\end{cor}
\begin{proof}
Let $\alpha=d\gamma$ be a non-zero and semi-positive $(1,1)$-form, and assume by contradiction that $\g$ admits a compatible balanced metric, with fundamental form $\omega$. Then, if $\dim(\g)=2n$,
\[
0=d\omega^{n-1}\wedge \gamma =d(\omega^{n-1}\wedge \gamma)+\omega^{n-1}\wedge d\gamma=\omega^{n-1}\wedge \alpha\,,
\]
where we used that every $(2n-1)$-form is closed, thanks to unimodularity. This would imply that $\alpha=0$, against our assumption, hence $\g$ cannot carry balanced metrics.
\end{proof}

Also, for future purposes, it will be useful to write down explicitly the SKT condition. The identity is well-known, see e.g. \cite{AN22,EFV12}, therefore we omit the proof.

\begin{lem}\label{SKTmaledetto}
Let $ \f g$ be a Lie algebra with a complex structure $J$. A Hermitian metric $g$ on $\f g$ is SKT if and only if
\[
\begin{aligned}
0&=g([J[X, JX], JY], JY)+g([J[X, JX], Y],Y)+2g([Y,JY],[X, JX])\\
&\quad +g([J[X, Y], X], JY)-g([J[X, Y], Y], JX)-g([J[X, JY], X], Y)\\
&\quad -g([J[X, JY], JY],JX)+g([J[JX, Y], JX], JY)+g([J[JX, Y], Y],X)\\
&\quad -g([J[JX, JY], JX], Y)+g([J[JX, JY], JY], X)+g([J[Y, JY], JX], JX)\\
&\quad +g([J[Y, JY], X], X)-|[X, Y]|^2-|[X, JY]|^2-|[JX, Y]|^2-|[JX, JY]|^2\,,
\end{aligned}
\]
for all $X,Y\in \g$.
\end{lem}

Along our treatment it will become apparent that many structural properties of the Lie algebras taken into account are strictly tied to the curvature of the Chern connection. For this reason, we recall the expression of the \emph{first Chern--Ricci form} for  Hermitian metrics on a given Lie algebra.

\begin{prop}[\cite{V13}]\label{chernrigido}
Let $\f g$ be a Lie algebra with a complex structure $J$. The Chern--Ricci form $\rho$ of any left-invariant metric on $\f g$ can be expressed as
$$
\rho(X, Y)=-\frac12{\rm tr}(J\ad_{[X, Y]})+\frac12{\rm tr}(\ad_{J[X, Y]})\,, \qquad X, Y\in \f g\,.
$$
In particular,  $\rho$ does not depend on the choice of the metric, and $\rho=d\sigma$, where $\sigma $ is the \emph{Koszul form} defined as 
\[
\sigma(X):=\frac{1}{2}\mathrm{tr}(J \ad_X)-\frac{1}{2}\mathrm{tr}(\ad_{JX})\,, \qquad X\in \f g\,.
\] 
\end{prop}

For future reference, we  also denote with $\eta$ the following form: 
\begin{equation}\label{defn_eta}
\eta(X, Y):=-\frac12\tr(J\ad_{[X, Y]})\,, \qquad X,Y\in\f g\,.
\end{equation}
It is clear that if $\f g$ is unimodular then $\rho=\eta$.  We observe that, as for $\rho$, the form $\eta $ is again $d$-exact. Indeed, $\eta=d\beta$ with 
\begin{equation}\label{eqn_beta}
\beta(X):=\frac12{\rm tr}(J\ad_X)\,, \qquad X\in \f g\,.
\end{equation}
It will also be useful for future purposes to have the following alternative expression for $\eta$. 
\begin{lem}\label{riccihol}
Let $\f g$ be a Lie algebra endowed with a complex structure $J$. Then, 
\begin{equation}\label{form_ricci}
\eta(X, JX)=-\frac14{\rm tr}([J,{\rm ad}_{JX}]^2)\,, \qquad X\in \f g\,.
\end{equation}
\end{lem}
\begin{proof}
We have  
$$
{\rm tr}(J{\rm ad}_{[X, JX]})=
-{\rm tr}([J, {\rm ad}_{JX}]{\rm ad}_{X})={\rm tr}([J, {\rm ad}_{JX}]J[J,{\rm ad}_{X}])+ {\rm tr}([J, {\rm ad}_{JX}]J{\rm ad}_XJ)\,.
$$
On the other hand,
$$
{\rm tr}([J, {\rm ad}_{JX}]J{\rm ad}_XJ)={\rm tr}(-{\rm ad}_{JX}J{\rm ad}_{X}+ {\rm ad}_{JX}{\rm ad}_{X}J )={\rm tr}(J[{\rm ad}_{JX},{\rm ad}_{X}])=-\tr(J\ad_{[X,JX]})\,,
$$
which in particular, using \eqref{integrability}, gives 
$$
\eta(X,JX)=-\frac{1}{4}{\rm tr}([J, {\rm ad}_{JX}]J[J,{\rm ad}_{X}])=-\frac14{\rm tr}([J, {\rm ad}_{JX}]^2)\,,
$$
as claimed.
\end{proof}
 We conclude the section with the following interesting fact.
\begin{lem}\label{Lem:balancedcsf}
Let $\f g$ be a unimodular Lie algebra endowed with a complex structure $J$. Then, any balanced Hermitian metric  is Chern-scalar flat.
\end{lem}
\begin{proof}
By Proposition \ref{chernrigido} we know that $\rho=d \sigma$. 
Now, given  a balanced metric $\omega$, we have 
$$
s^{{\rm Ch}}(\omega)\frac{\omega^n}{n!}=\rho\wedge \frac{\omega^{n-1}}{(n-1)!}=d\sigma\wedge \frac{\omega^{n-1}}{(n-1)!}=0
$$ using that $d\omega^{n-1}=0$ and unimodularity.
\end{proof}

\section{Abelian complex structures}\label{secabel}
 
Given  a $2$-step solvable  Lie algebra $\g $,  we have seen in Lemma \ref{Lem2stepideal} that $\g'_J:=\f g'\cap J\f g'$ is always an ideal of $\f g'+J\f g'$. Hence,  we can consider the quotient $(\f g'+J\f g')/\f g_J'$ which yields a  Lie algebra with abelian complex structure. With this in mind, we proceed with our investigation by studying abelian complex structures.

The following simple lemma supplies us with a very useful identity that we will use repeatedly.

\begin{lem}
Let $\f g $ be a  Lie algebra endowed with an abelian complex structure $J$. Then 
\begin{equation}\label{penis}
\ad_Y\ad_X+\ad_{JX}\ad_{JY}=\ad_{J[JX, Y]}\,, \qquad X, Y \in \f g\,.
\end{equation}
\end{lem}
\begin{proof}
By simply using Jacobi identity and abelianity of $J$,  we have that 
$$
\ad_{JX}\ad_{JY}=-\ad_{JX}\ad_YJ=-\ad_{[JX, Y]}J-\ad_Y\ad_{JX}J=\ad_{J[JX,Y]}-\ad_Y\ad_X\,, 
$$
as claimed.
\end{proof}

By means of \eqref{penis}, we are able to prove the following result, whose importance will become clear quickly.
\begin{prop}\label{propKilling}
Let $\f g $ be a Lie algebra endowed with an abelian complex structure $J$. Then,
\begin{enumerate}
\item $\f n_J$ is a $J$-invariant ideal, i.e.  $\f n_J$ is the maximal $J$-invariant nilpotent ideal of $\f g$;
\item $B^{1,1}(JX , Y)=-\eta(X, Y)$, for any $X, Y\in\f g $, and so it is $d$-exact; 
    \item  for any $Z\in \f n , X, Y\in\f g$, we have 
    $$
    B^{1,1}([JZ, X], Y)=B^{1,1}([JZ, Y], X)\,, \qquad B^{1,1}([Z, X], Y)=B^{1,1}(J[Z, JY], X)\,;
    $$
    \item if $\f g=\f n +J\f n$, then $\ker \eta=\ker B\cap J \ker B =\f n_J$;
    \item if $\f g=\f n +J\f n$ and $\f k $ is a $J$-invariant ideal of $ \f g$, then $\f k ^{\perp}$ is a $J$-invariant  ideal, where $(\cdot)^{\perp}$ is the orthogonal complement with respect to $B^{1,1}$. Furthermore, $B_{\f k}^{1,1}=B^{1,1}\vert_{\f k\times \f k}$.
\end{enumerate}
\end{prop}
\begin{proof}
Most of the proposition hinges on the identity \eqref{penis}.
Up to complexification, we can use Lie's Theorem and assume that $\ad_Z$ is upper triangular for all  $Z\in\f g$, after choosing a suitable basis of $\f g$. Now, take $X \in\f n_J$. Then $\ad_X $ and $ \ad_{JX}$ are strictly upper triangular. For any $Y\in \g $, this forces $\ad_Y\ad_X$, $\ad_{JX}\ad_{JY}$, and hence also $\ad_{J[JX, Y]}$, to be strictly upper triangular, and thus nilpotent. Thanks to this $J[JX, Y]\in \f n$ and consequently $[JX, Y]\in\f n_J$,  proving  that $\f n_J$ is an ideal. The last part of \emph{(1)} is straightforward.

As for \emph{(2)}, we use \eqref{penis} again.  Indeed, for any $X,Y\in \g$, 
$$
B^{1,1}(JX , Y)=\frac12{\rm tr}(\ad_{JX}\ad_Y-\ad_{X}\ad_{JY})=-\frac12{\rm tr}(\ad_{J[X, Y]})=\frac12{\rm tr}(J\ad_{[X, Y]})=-\eta(X,Y)=-d\beta(X,Y)\,,
$$
where  $\beta$  is the form defined in \eqref{eqn_beta}, which proves the claim.

To show \emph{(3)}, we exploit the closure of $\eta$. First of all, we notice that $B^{1,1}(\f n, J \f n)=0$, since $\f n \subseteq \ker B$. Now,  let $X, Y\in\f g$ and $Z\in\f n$, we have 
$$
0=d\eta(JX,JY, Z)=B^{1,1}(J[JX, JY], Z)-B^{1,1}(J[JX, Z], JY)+B^{1,1}(J[JY, Z], JX)\,.
$$
Using that $[JX, JY]\in\f n$, we get
$$
0=-B^{1,1}([JX, Z], Y)+B^{1,1}([JY, Z], X)=B^{1,1}([X, JZ], Y)-B^{1,1}([Y, JZ], X)\,,
$$
which implies  
$$
B^{1,1}([JZ, X], Y)=B^{1,1}([JZ, Y], X)\,, 
$$
as claimed. On the other hand,
$$
B^{1,1}([Z, X], Y)=B^{1,1}([JZ, JX], Y)=B^{1,1}([JZ, Y], JX)=B^{1,1}(J[Z, JY], X)\,, 
$$
concluding the proof of \emph{(3)}.

As for \emph{(4)}, we observe that 
$
\ker B\cap J \ker B\subseteq \ker \eta\,.
$
Viceversa, fix  $X\in \ker \eta$ and $Y\in \f n$. By expanding the relation $\eta(X, Y)=0$ in terms of $B$, using part \emph{(2)}, we obtain $ B(X, JY)=0.$ Together with $\f n \subseteq \ker B$, this implies that $X\in \ker B$. On the other hand, expanding $\eta(X, JY)=0$, we get 
$
B(JX, JY)=0\,,  
$
which yields $JX\in \ker B$. Thus $X \in \ker B \cap J \ker B$, giving $\ker B \cap J \ker B =\ker \eta$. 
% Now, we want to show that $\ker \eta = \f n_J$.
Finally, to show that $\ker \eta = \f n_J$, let  $X\in \ker \eta$ and write $X=Z+JY$, where $Y, Z\in \f n$. It follows that $JY\in \ker B$. Thus, we have
$$
{\rm tr}(\ad_{JY}^2)=0\,.
$$
We shall prove that ${\rm tr}(\ad_{JY}^{k+1})=0 $ for all $k\ge 2$.  Using \eqref{penis}, we infer
$$
\begin{aligned}
\ad_{JY}^{k+1}=\ad_{JY}^{k-1}\ad_{J[JY, Y]}-\ad_{JY}^{k-1}\ad_Y^2\,.
\end{aligned}
$$
Now, $\ad_Y$ is nilpotent. Hence, as above, $\ad_{JY}^{k-1}\ad_Y^2$ is nilpotent and therefore 
$$
{\rm tr}(\ad_{JY}^{k+1})={\rm tr}(\ad_{JY}^{k-1}\ad_{J[JY, Y]})\,.
$$
Applying \eqref{penis} once again, we obtain 
$$
\ad_{JY}^{k-1}\ad_{J[JY, Y]}=\ad_{JY}^{k-2}\ad_{J[JY,[JY, Y]]}-\ad_{JY}^{k-2}\ad_{[JY, Y]}\ad_Y\,,
$$
and thus
$$
{\rm tr}(\ad_{JY}^{k+1})
={\rm tr}(\ad_{JY}^{k-2}\ad_{J[JY,[JY, Y]]})\,.
$$
Iterating this process, we eventually obtain 
$$
{\rm tr}(\ad_{JY}^{k+1})={\rm tr}(\ad_{JY}\ad_{J[JY,[JY, \cdots [JY, Y]]\cdots]})=0\,,
$$
since $ JY\in \ker B$. This ultimately leads to the fact that $\ad_{JY}$ is nilpotent, and  hence $\ad_{X}$ is nilpotent, proving that $\ker \eta \subseteq \f n$. On the other hand, since clearly $J \ker \eta =\ker \eta $, we deduce that $\ker \eta \subseteq \f n_J$. Conversely, it is well-known that $\f n \subseteq \ker B$ and so $\f n_J\subseteq \ker B \cap J \ker B=\ker \eta$.

Finally, we prove \emph{(5)}. Let $\f k $ be a $J$-invariant ideal of $\f g$. Being $B^{1,1}$ of type $(1,1)$ guarantees straightforwardly that $\f k ^{\perp}$ is $J$-invariant. Let $X\in \f k $, $Y\in \f k ^{\perp}$ and $Z\in \f g$ and  write $Z=Z'+JZ''$ with $Z',Z''\in \f n$. Thus, applying part \emph{(3)}, we get
\[
\begin{split}
B^{1,1}([Z,Y ], X)&=B^{1,1}([Z',Y ], X)+B^{1,1}([JZ'',Y], X)=B^{1,1}(J[Z',JX],Y)+B^{1,1}([JZ'', X],Y)=0\,,
\end{split}
\]
since $\f k $ is a $J$-invariant ideal. This concludes the first part of \emph{(5)}. The second part is a trivial consequence of the facts that $B_{\f k }=B_{\f g }|_{\f k\times \f k}$ and  that $J\f k =\f k $.
\end{proof}

An important class of Lie algebras admitting  abelian complex structures is that of affine Lie algebras, see \cite{ABD12, BD04}. In what follows, we will be interested only in those affine Lie algebras constructed from commutative and associative $\R$-algebras.

\begin{defn}\label{defiaffi}
Let $A$  be a  finite-dimensional, commutative and associative $\R$-algebra.  The affine Lie algebra $\f{aff}(A)$ is $A\oplus A$ with Lie bracket given by 
$$
[(a, b), (a', b')]:=(0, ab'-a'b)\,, \quad a, b, a', b'\in A\,,
$$
and complex structure defined by $J(a, b):=(b, -a)$.
\end{defn}

It is straightforward to notice that if $\f g=\f{aff}(A)$, for some finite-dimensional, commutative and associative $\R$-algebra $A$, then $\f g_J'=0$.
 
We prove the following basic lemma which will be useful later. 
\begin{lem}\label{lemmaffini}
Let $A$ and $A'$ be  finite-dimensional, commutative and associative $\R$-algebras.
\begin{enumerate}
\item  Any subalgebra of $A$ induces a $J$-invariant subalgebra of $\f{aff}(A)$. Any ideal of $A$ induces a $J$-invariant ideal of $\f{aff}(A).$
\item Any $J$-perfect subalgebra of $\f{aff}(A)$ corresponds to a subalgebra of $A$. Any $J$-perfect ideal of $\f{aff}(A)$ corresponds to an ideal of $A$.
\item  Assume $A=S\oplus I$, where  $I$ is an ideal and $S$ is a subalgebra of $A$. Then $\f{aff}(A)= \f{aff}(S)\ltimes \f{aff}(I)$.
\item Let  $\varphi\colon A \to A'$  be a homomorphism of $\R$-algebras. Then there exists $\f{aff}( \varphi )\colon \f{aff}(A)\to \f{aff}(A')$ holomorphic Lie algebra homomorphism such that $\ker \f{aff}(\varphi)=\f{aff}(\ker \varphi)$. In particular, if $\varphi$ is an isomorphism, then $\f{aff}( \varphi)$ is a holomorphic isomorphism.
\item  Let $I$ be an ideal of $A$. Then $\f{aff}(A/I)\simeq \f{aff}(A)/\f{aff}(I).$
\end{enumerate}
\end{lem}
\begin{proof}
We start with \emph{(1)}. 
Let $S$ be a subalgebra of $A$ and consider  $\f{aff}(S)\subseteq \f{aff}(A)$. Since $S$ is a subalgebra,  we have that, for any $X=(a, b), Y=(a', b')\in\f{aff}(S) $ with $a,b,a',b'\in S$, 
$$
 [X, Y]=(0, ab'-a'b)\in \f{aff}(S)\,, 
$$
which implies that $\f{aff}(S)$ is a subalgebra of $\f{aff}(A)$, which is clearly $J$-invariant.  The statement for ideals follows from the same arguments. 

As for \emph{(2)}, we fix  a $J$-perfect subalgebra $\f s$ of $\f{aff}(A)$. We consider $S:=p_2\f s'\subseteq A$, where $p_2\colon \f{aff}(A)\to A$ is the projection onto the second factor.  Now, fix $a, b\in S$, then there exist $X,Y\in \f s'$ such that $X=(0,a)$ and $Y=(0,b)$. Then, 
$$
\f s'\ni[JY,X]=(0,ab)\,,
$$
which shows that $ab\in S$. We conclude that $S$ is a subalgebra of $A$.

Now, we prove \emph{(3)}. 
Using \emph{(1)}, $\f{aff}(S)$ and $\f{aff}(I)$ are, respectively, a $J$-invariant subalgebra and a $J$-invariant ideal of $\f{aff}(A)$. The claim then follows from the assumption that $A=S\oplus I $.

To prove \emph{(4)},   we  consider 
$$
\f{aff}( \varphi):=\varphi \oplus \varphi\colon \f{aff}(A)\to \f{aff}(A')\,, \quad \f{aff}(\varphi)(a, b):=(\varphi(a), \varphi(b))\,.
$$
Using the definition of the  bracket and the complex structure on affine Lie algebras, it is straightforward to check that $\f{aff}(\varphi)$ is a holomorphic Lie algebra homomorphism. 
Now, $\ker \f{aff}( \varphi)=\ker \varphi\oplus \ker \varphi= \f{aff}(\ker \varphi)$, as claimed. 

Finally, by \emph{(4)}, the projection onto the quotient $\pi\colon A\to A/I$ induces a holomorphic Lie algebra homomorphism $\f{aff}( \pi)\colon \f{aff}(A)\to \f{aff}(A/I)$ such that $\ker \f{aff}(\pi)=\f{aff}(\ker \pi)=\f{aff}(I)$. On the other hand, since $I$ is an ideal, \emph{(1)} implies that $\f{aff}(I) $ is a $J$-invariant ideal of $\f{aff}(A)$, hence we may consider the canonical projection onto the quotient $\pi'\colon \f{aff}(A)\to \f{aff}(A)/\f{aff}(I)$. Since $\ker \pi'=\ker\f{aff}( \pi)$ and $\f{aff}( \pi)$ is surjective, we can find a holomorphic isomorphism $f\colon \f{aff}(A)/\f{aff}(I)\to \f{aff}(A/I)$ such that $\f{aff}( \pi)=f\pi'$, which proves the claim.
\end{proof}

\subsection{Semisimple abelian complex structures}
We discovered in Proposition \ref{propKilling} that, on Lie algebras of the form $\f g=\f n+ J\f n$  endowed with an abelian complex structure, the kernel of $B^{1,1}$ is an ideal. With this piece of knowledge it becomes possible to characterise non-degeneracy of $B^{1,1}$ as follows.

\begin{thm}\label{Semisimple}
Let $\f g$ be a Lie algebra with an abelian complex structure $J$. Then the following are equivalent
\begin{enumerate}
    \item $J$ is semisimple;
    \item $B^{1,1}$ is non-degenerate and $\f g=\f n(\f g)\oplus J\f n(\f g)$  as vector spaces; 
    \item $\f g$ is the direct sum of  $J$-simple  ideals $\f k_i$. 
\end{enumerate}
In particular, if any of the above holds, $\f g=\g'\rtimes J\g'$ and $J$-invariant ideals and quotients of $\f g$ are $J$-semisimple. 
\end{thm}
\begin{proof}
Assume that \emph{(1)} holds. By Proposition \ref{propKilling}, part \emph{(1)}, we have $\f n(\f g)_J=0$. Set $\f k=\f n(\f g)\oplus J \f n(\f g)$ and consider $\f k^{\perp}$, where $(\cdot)^{\perp} $ denotes the orthogonal complement with respect to $B^{1,1}$. Clearly, $\f k^{\perp}$ is $J$-invariant. Applying Lemma \ref{lem_nilideali}, we obtain $\f k=\f n(\f k)\oplus J\f n(\f k)$  and, Proposition \ref{propKilling}, part \emph{(4)}, gives  $\f k \cap \f k^{\perp}=\ker B^{1,1}_{\f k}= \f n(\f k)_J=0$. Now, since $d\eta=0$ and $\f k$ is an ideal, for any $X, Y \in\f k^{\perp}$ and $ Z\in\f k$ we have
\begin{equation}\label{detauguale0}
0=d\eta(X, Y, Z)=B^{1,1}(J[X, Y], Z)-B^{1,1}(J[X, Z], Y)+B^{1,1}(J[Y, Z], X)=B^{1,1}(J[X, Y], Z)\,,
 \end{equation}
which implies $[X, Y]\in\f k^{\perp}\cap \f k =0$.  
Hence $\f  k^{\perp}$ is abelian. Instead, choosing $X\in\f k^{\perp}, Y, Z \in\f n(\f g)  $ and  using Proposition \ref{propKilling}, part \emph{(3)}, we deduce 
$$
B^{1,1}([JZ, X], Y)=B^{1,1}([JZ, Y], X)=0\,, \quad B^{1,1}([Z, X], Y)=B^{1,1}(J[Z, JY], X)=0\,,
$$
since $[JZ, Y], J[Z, JY]\in\f k $.  This allows us to infer that $[\f k^{\perp}, \f k ]\subseteq \f k^{\perp}\cap \f k=0.$  Using \eqref{detauguale0} again, but now with $X\in\f k ^{\perp}$, $Y\in\f g$ and $Z\in\f k $, we conclude that $\f k^{\perp}$ is a $J$-invariant abelian ideal of $\f g$. By semisimplicity, this forces $\f k^{\perp}=0$. In particular, $\ker B^{1,1}\subseteq \f k^{\perp}=0$, which implies that $B^{1,1}$ is non-degenerate. Since $B^{1,1}_{\f k}$ is also non-degenerate we deduce that $\f g=\f n(\f g)\oplus J\f n(\f g)$ and \emph{(2)} follows. 

Assume now that \emph{(2)} holds. We  adapt the proof of \cite[Theorem 1.51]{K02} to prove \emph{(3)}. If $\f g $ has no proper $J$-invariant ideals, the claim is immediate. Let us assume that $\f g $ has a proper $J$-invariant ideal, hence we can consider the minimal non-zero $J$-invariant ideal $\f k $. By Lemma \ref{lem_nilideali}  and  the minimality of $\f k $, we obtain $\f k=\f n(\f k )\oplus J  \f n(\f k )$.
By Proposition \ref{propKilling}, part \emph{(5)}, we know that $\f k ^{\perp} $ is a $J$-invariant ideal. Moreover, $\f k\cap \f k^{\perp}=\ker B_{\f k}^{1,1}=\f n(\f k)_J=0$, thanks to Proposition \ref{propKilling}, part \emph{(4)}, and the fact that $\f n(\f g)_J=0$. In particular, $\f g =\f k \oplus \f k^{\perp}$  as a Lie algebra.
Now observe that, if $\f l \subseteq \f k  $ is a $J$-invariant ideal, then $[\f l , \f k^{\perp}]\subseteq \f k\cap \f k ^{\perp}=0$, so $\f l$ is also a $J$-invariant ideal  of $\f g$. By the minimality of $\f k $,  it follows that $ \f l $ is either $0$ of $\f k $. Hence, $ \f k $ is $J$-simple. For the same reason,  any $J$-invariant ideal of $\f k^{\perp}$ is also a $J$-invariant ideal of $\f g $. Therefore, $\f n(\f k^{\perp})_J=0$  and $\f k^{\perp}=\f n( \f k^{\perp})\oplus J\f n( \f k^{\perp})$. To see the last equality,  fix $X=N+JN'\in \f k^{\perp}$, for some  $N, N'\in\f n(\f g)$ and  $Y\in \f n(\f k )$. Since  $B^{1,1}(\f n(\f g), J\f n(\f g ))=0$, we infer  
$$
0=B^{1,1}(X, Y)=B^{1,1}(N, Y)+ B^{1,1}(JN', Y)=B^{1,1}(N, Y)\,.
$$
Thus,  $N\in (\f n(\f k )\oplus J\f n(\f k))^{\perp}=\f k^{\perp} $. From this, it follows that   $\f k^{\perp}\subseteq \f n(\f k^{\perp})\oplus J\f n(\f k^{\perp}) $. Therefore, $B^{1,1}_{\f k ^{\perp}}$ is non-degenerate, using Proposition \ref{propKilling} part \emph{(1)} and \emph{(4)}. We can then repeat the procedure on $\f k^\perp$ and conclude after finitely many steps.

Finally, assume \emph{(3)}, namely  
$$
\f g=\bigoplus_{i=1}^{k}\f k_i\,,
$$
where $\f k_i$ is  a $J$-simple ideal of $\f g$. Of course, for any $i=1, \ldots, k$, the projection $\pi_i\colon \f g \to \f k_i$ is a holomorphic Lie algebra homomorphism. Hence, if $\f k\subseteq \f g $ is a $J$-invariant ideal of $\f g $, then $\pi_i(\f k )$ is a $J$-invariant ideal of $\f k_i$. Therefore, $\pi_i(\f k ) $ is either $0$ or $\f k_i$. If $\pi_i(\f k)=\f k_i$ then $\f k_i\subseteq \f k .$ Indeed, 
$$
\f k_i'=[\f k_i, \f k_i]_{\f k_i}=[\f k_i, \pi_i(\f k )]_{\f k_i}=[\f k_i, \f k ]\subseteq \f k\,. 
$$
Since $\f k $ is $J$-invariant  and $\f k_i$ is $J$-simple, hence $J$-perfect, we also have 
$\f k_i=\f k_i'\rtimes J\f k_i'\subseteq \f k $. Consequently,
$$
\f k=\bigoplus_{\f k_i\subseteq \f k }\f k_i\,.
$$
But now, it is clear that $\f k$ cannot be $J$-solvable unless $\f k=0$. 

In addition,
$$
\f g'=\left[\f g, J\f g \right]=\bigoplus_{i=1}^k[\f k_i, J \f k_i]=\bigoplus_{i=1}^k \f k_i'\,.
$$
Therefore, 
$$
\f g'\rtimes J \f g'=\bigoplus_{i=1}^k\f k_i'\rtimes J\f k_i'=\f g \,,
$$
showing that $\f g$ is $J$-perfect.
\end{proof}

\begin{prop}
Let $\f g$ be a  Lie algebra endowed with a semisimple abelian complex structure $J$. Then the group ${\rm Aut}(\f g , J)$ of holomorphic automorphisms of $\f g$ is discrete. As a consequence, $\f g $ has no holomorphic vectors.
\end{prop}
\begin{proof}
It is known that every $D\in {\rm Der}(\f g )$ satisfies $D\f g\subseteq \f n $. If $D\in {\rm Der}(\f g , J):=\{D\in {\rm Der}(\f g) \,\, |\, \,  DJ =JD\}$, then $D\f g \subseteq \f n _J=0$, due to semisimplicity of $\f g$. 
Hence, $D=0$, meaning that ${\rm Der}(\f g, J)=\{0\}$. The claim now follows from the well-known fact that ${\rm Der}(\f g, J)={\rm Lie}({\rm Aut}(\f g, J))$.  Finally, if $X \in \f g $ is a holomorphic vector field, then $\ad_X\in {\rm Der}(\f g, J)$, which implies $X\in \f z(\f g)$. On the other hand, the centre of $\f g$ is a $J$-invariant abelian ideal, so $X=0$.
\end{proof}

Let $\f g$ be a $J$-perfect and $J$-semisimple,  2-step solvable Lie algebra. Then either $\g'_J=0$ or $\g'_J=\f g$. The second case cannot occur, because $\f g$ is solvable. Therefore $\f g$ is a direct sum of $J$-simple ideals by Theorem \ref{Semisimple}. However, the next example shows that if $\f g$ is not $J$-perfect then it might not have such structure. This is in striking contrast with the classical fact that a semisimple Lie algebra is always the direct sum of simple ideals.

\begin{es}\label{ssnonabelian}
Consider the $2$-step solvable Lie algebra $\g=\langle e_1,e_2,e_3,e_4,e_5,e_6\rangle$ with only non trivial brackets:
\[
[e_1,e_5]=[e_2,e_6]=[e_3,e_4]=e_4\,, \qquad [e_3,e_5]=e_5\,, \qquad [e_3,e_6]=e_6\,.
\] and complex structure given by $$Je_1=e_2\,, \quad Je_3=e_4\,,\quad Je_5=e_6\,.
$$
As a Lie algebra, $\g$ is isomorphic to $\f s_{6,182}$ in the notation of \cite{SW14}.
It is fairly easy to see that $\g'+J\g'=\langle e_3,e_4,e_5,e_6 \rangle$ is the only non-trivial $J$-invariant ideal of $\g$. However, $\g'+J\g'$ is $J$-perfect and therefore cannot be $J$-solvable, implying that $\g$ is $J$-semisimple, even though $J$ is not abelian. 
\end{es}

Also, let us observe that when the solvability step is higher than $2$ a $J$-simple Lie algebra might have degenerate $B^{1,1}$. Indeed, Example \ref{exsimplerigid3step} is easily checked to be $J$-simple, however
\[
B^{1,1}=-\left( e^1\otimes e^1 + e^4 \otimes e^4\right)\,,
\]
in particular, $\ker(B^{1,1})=\langle e_2,e_3 \rangle$.

\subsection{Simple complex structures}
In view of Theorem \ref{Semisimple}, we now turn our attention to the study of $J$-simple $2$-step solvable Lie algebras $\f g$. We will fully classify them, with two different proofs: the first of a more Lie-theoretic flavour while the second of a more algebraic flavour. Let us introduce the candidates before stating the theorem itself.

\begin{es}\label{aff}
The Lie algebra of affine motions of $\R$ is denoted by $\f{aff}(\R)$ and it is the unique non-abelian, solvable $2$-dimensional Lie algebra: 
$$
[JX, X]=X\,.
$$
One can easily see that, being $\f{aff}(\R)$ completely solvable, $B^{1,1}>0$ and hence $-\eta$ is an exact K\"ahler metric on $\f{aff}(\R)$.

The Lie algebra of affine motions of $\C$ will be denoted by $\f{aff}(\C)$ and it is the $4$-dimensional Lie algebra defined by the following structure equations:
$$
[JX, X]=X, \quad [JX, Y]=[JY, X]=Y, \quad [JY, Y]=-X.
$$
Thanks to the classification in \cite{S90}, $\f{aff}(\C)$ is the unique $4$-dimensional non-completely solvable Lie algebra admitting an abelian complex structure. Moreover, one can easily verify that $B^{1,1}$ in this case has signature $(2, 2,0)$. We note that $\f{aff}(\C)$ also carries another, non-equivalent, abelian complex structure which is solvable, see e.g. \cite[Corollary 2.7]{ABD11}. Finally, it is easy to see that $\f{aff}(\C)$ cannot carry any SKT metric. 

Throughout the paper, whenever we write $\f{aff}(\R)$ and $\f{aff}(\C)$ we shall intend them as Lie algebras equipped with the respective abelian complex structures defined here.
\end{es}
 The next theorem gives the full classification of $J$-simple $2$-step solvable Lie algebras.

\begin{thm}\label{affines}
Let $\f g $ be a $2$-step solvable Lie algebra endowed with a simple complex structure $J$. Then, either $\f g \simeq \f{aff}(\R)$ or $\f g \simeq \f{aff}(\C)$.
\end{thm}
\begin{proof}
Let us begin observing that, since $\f g $ is $2$-step solvable, the family $\{\ad_{JV}|_{\f g'}\}_{V\in \f g'}$ consists of commuting operators. Therefore, we can find $Z\in\f g'\otimes \C$ which is a common eigenvector of  $\{\ad_{JW}|_{\f g'\otimes \C}\}_{W\in\f g' \otimes \C}$, i.e., for any $W\in \f g' \otimes \C $, 
\begin{equation}\label{common}
[JW,Z ]=\lambda(W)Z\,,  \qquad \lambda\in(\f g'\otimes \C)^*\,.
\end{equation} 
Hence, in particular, we obtain 
\begin{equation}\label{commonZ}
[JZ, Z]=\lambda(Z)Z\,.
\end{equation}
Now, we have two possibilities: either $Z\in\f g'$ or $Z$ is genuinely complex. Let us assume that $Z\in\f g' $. This allows us to infer that 
$$
\f k_1 := \langle Z, JZ\rangle
$$
is a non-zero $J$-invariant ideal of $\f g $. By hypothesis, $\f k_1 = \f g.$ Moreover, $\lambda(Z)\ne 0 $ otherwise $\f g$ would be abelian against simplicity of  $J$. Hence, up to renormalisation, we can suppose that $\lambda (Z)=1$ which gives straightforwardly that $\f g \simeq \f{aff}(\R )$.
It remains to understand the case in which $Z$ is complex. Let us write $Z=X+\sqrt{-1}Y$, with $X, Y\in\f g'$. From \eqref{common}, it follows that, for any $V\in \f g' $, 
  \begin{equation}\label{affc}
  [JV, X]=\lambda_1(V)X-\lambda_2(V)Y\,, \qquad [JV, Y]=\lambda_2(V)X+ \lambda_1(V)Y\,,
  \end{equation}
  where $\lambda_1=\Re(\lambda)$ and $\lambda_2=\Im(\lambda)$. This, in particular, guarantees that
  \[
  \f k_2:= \langle X, Y, JX, JY \rangle
  \]
  is a non-zero $J$-invariant ideal of $\f g$ forcing $\f k_2 =\f g$, by $J$-simplicity. Next, we show that  $\f g\simeq \f{aff}(\C).$ Applying \eqref{affc} to $X$ and $Y$, we get 
  \begin{equation}\label{hiddeneqaffC}
  \begin{aligned}
  \relax
  [JY, Y]=&\,\lambda_2(Y)X+\lambda_1(Y)Y\,, \qquad [JY, X]=\lambda_1(Y)X- \lambda_2(Y)Y\\
  [JX, Y]=&\, \lambda_2(X)X+\lambda_1(X)Y\,, \qquad [JX, X]=\lambda_1(X)X- \lambda_2(X)Y\,.
  \end{aligned}
\end{equation}

Now, since $J$ is abelian then $[JY, X]=[JX, Y]$, hence $\lambda_2(X)=\lambda_1(Y)$ and $\lambda_1(X)=-\lambda_2(Y)$, guaranteeing that $[JY, Y]=-[JX, X].$ Rewriting \eqref{commonZ} using these relations shows that $\lambda(Z)= 2\lambda(X)$. Therefore, $\lambda(Z)\ne 0 $, since otherwise $\lambda_1(X)=\lambda_2(X)=0$ and $\f g $ would be abelian, against the fact that $J$ is simple. Then, up to renormalisation, we can assume $\lambda(X)=1$. Thus, we can rewrite \eqref{hiddeneqaffC} as
$$
  [JY, Y]=-X\,, \quad [JY, X]=Y\,, \quad 
  [JX, Y]=Y\,, \quad [JX, X]=X\,,
$$
which yields the desired isomorphism.
\end{proof}

\begin{rmk}
As anticipated, Theorem \ref{affines} can be proved in a more algebraic manner.  Thanks to \cite[Theorem 3.5]{ABD12}, any Lie algebra $\f g = \f g'\rtimes J\f g' $ endowed with an abelian complex structure is holomorphically isomorphic to $\f{aff}( A)$, where $ A$ is a finite-dimensional, associative, commutative $\R$-algebra, see Definition \ref{defiaffi}.  In our hypotheses, the algebra $A$ is  given by  $ A=(\f g', \cdot)$ with 
$$
X\cdot Y:=[JX, Y]\,, \qquad X, Y\in \f g'\,.
$$
Hence, if $\f g=\f g'\rtimes J \f g' $  is $J$-simple, then the algebra $ A$ is simple, using Lemma \ref{lemmaffini}, part \emph{(1)}. Therefore, we can use Artin-Wedderburn Theorem, see \cite[Theorem, Section 3.5]{Pie82}, and Frobenius' Theorem, see for instance \cite{Palais68}, which imply that the only simple, commutative and associative $\R$-algebras are either $\R$ or $\C$, yielding the desired claim.
\end{rmk}

The next result provides an analogue of the Cartan decomposition of a real semisimple Lie algebra in our context, see \cite[Chapter VI]{K02}. 

\begin{cor}\label{corsymmetricdecom}
Let $\f g$ be a  Lie algebra endowed with a semisimple abelian complex structure $J$. Then, there exist $p, q\ge 0,$ $p+q\ge 1$ such that
\begin{equation}\label{sssss}
\f g=\f{aff}(\R)^p\oplus \f{aff}(\C)^q\,.
\end{equation}  
In particular
\begin{enumerate}
\item there exists a $J$-invariant subalgebra $\f g_+$ and a $J$-invariant subspace $\f g_-$ such that 
$$
\f g=\f g_+\oplus \f g_-\,,\quad [\f g_+, \f g_+]\subseteq \f g_+ \,, \quad [\f g_+, \f g_-]\subseteq \f g_-\,, \quad [\f g_-, \f g_-]\subseteq \f g_+\,.
$$
As a consequence, there exists an involution $s \in {\rm Aut}(\f g , J)$ such that $B_s^{1,1}(X, Y):=B^{1,1}(X, sY)$ defines a Hermitian metric.
\item $\ad_{JX}^{t_s}=\pm \ad_{JX}$, for all  $X\in \f g_{\pm}\cap \f g'$, where $(\cdot)^{t_s}$ denotes the transpose with respect to $B^{1,1}_s$.
\end{enumerate}
\end{cor}
\begin{proof}
By Theorem \ref{Semisimple} and Theorem \ref{affines} we immediately infer \eqref{sssss}.\\
To prove \emph{(1)}, we observe that for $\f{aff}(\R)$ and $\f{aff}(\C)$ we may choose 
\begin{align*}
\qquad  \qquad   \f{aff}(\R)_+&= \f{aff}(\R) \,, &\f{aff}(\R)_-&= 0 \,,\qquad \qquad   \\
\f{aff}(\C)_+&= \langle X,JX \rangle \,, & \f{aff}(\C)_-&= \langle Y, JY \rangle \,,
\end{align*}
where we retain the notations of Example \ref{aff}. It is easy to check that these subspaces satisfy the desired Lie bracket relations. Hence, thanks to \eqref{sssss}, it suffices to choose 
$$
\f g_+:=\f{aff}(\R)^p\oplus \f{aff}(\C)_+^q\,, \qquad \f g_-:=\f{aff}(\C)_-^q\,.
$$
Furthermore, we define 
$$
s|_{\f g_+}={\rm Id}_{\f g_+}\,, \quad s|_{\f g_-}=-{\rm Id}_{\f g_-}\,.
$$
Due to the bracket relations above and that $J\f g_{\pm}=\f g_{\pm}$, we obtain $s\in {\rm Aut}(\f g, J)$, and clearly $s^2={\rm Id}$. 
Next, it is easy to see that $B^{1,1}_s$ is of type $(1,1)$, since $[s, J]=0$. Moreover, using again that $[s, J]=0$ and $s\in \mathrm O(B)$, we have, for any $X, Y\in\f g$, 
$$
\begin{aligned}
B^{1,1}_s(X, Y)= \frac12\left(B(X, sY)+ B(JX, JsY)\right)
=\frac12\left(B(Y, sX)+ B(JY, JsX)\right)=B^{1,1}_s(Y, X)\,,
\end{aligned}
$$
giving us that $B_s^{1,1}$ is symmetric. 
To see positive definiteness, we note that  the decomposition of $\f g $ in $J$-simple ideals is  orthogonal with respect to  $B_s^{1,1}$. Hence, it is sufficient to prove the claim on each $J$-simple ideal. Since $B_{\f{aff}(\R)}^{1,1}>0$ and $s_{\f{aff}(\R)}={\rm Id}$, it is enough to prove the statement on $\f {aff}(\C)$, for which it is easy to see that 
$$
\qquad \, B^{1,1}_{\f{aff}(\C)}|_{\f{aff}(\C)_+\times \f{aff}(\C)_+}>0\,,\quad B^{1,1}_{\f{aff}(\C)}|_{\f{aff}(\C)_+\times \f{aff}(\C)_-}=0\,, \quad B^{1,1}_{\f{aff}(\C)}|_{\f{aff}(\C)_-\times \f{aff}(\C)_-}<0\,.
$$
Since $s|_{\f{aff}(\C)_-}=-{\rm Id}$, the claim follows straightforwardly.

As for \emph{(2)}, we have that, using  Proposition \ref{propKilling}, part \emph{(3)}, for any $X\in \f g_{\pm}\cap \f g'$,  
$$
\begin{aligned}
B^{1,1}_s(\ad_{JX} Y, Z)=&\, 
B^{1,1}(Y, \ad_{JX}sZ)
= B^{1,1}(Y, s\ad_{JsX}Z)=\pm B^{1,1}_s(Y, \ad_{JX}Z)\,,
\end{aligned}
$$
concluding the proof. \qedhere
\end{proof}

With this settled, whenever a Lie algebra $\f g$ is endowed with a semisimple and abelian complex structure $J$, then, thanks to \eqref{sssss}, we can introduce the following basis of $\f g$, which we call  \emph{standard}:
\begin{equation}\label{baseincredibile}
\mathcal B:=\{X_1, JX_1, \ldots,X_p, JX_p, X_{p+1}, JX_{p+1}, \ldots,  X_{p+q}, JX_{p+q}, Y_{p+1},JY_{p+1}, \ldots, Y_{p+q}, JY_{p+q}\}\,,
\end{equation}
 where $\{X_1, JX_1, \ldots, X_{p+q}, JX_{p+q}\}$ is a basis of $\f g_+$ whereas $\{Y_{p+1},JY_{p+1}, \ldots, Y_{p+q}, JY_{p+q}\}$ is a basis of $\f g_-$ such that $\langle X_i, JX_i\rangle\simeq \f{aff}(\R) $ for any $i=1,\ldots, p$, and $ \langle X_{p+j}, JX_{p+j}, Y_{p+j}, JY_{p+j}\rangle\simeq \f{aff}(\C)$ for any $j=1, \ldots, q$.
\subsection{Levi--Malcev decomposition for abelian complex structures}
In this subsection we will prove Theorem \ref{mainabelian}. Before proceeding, we introduce two tools that will be useful in what follows.

\begin{defn}\label{defn_U}
Let $\f g$ be a  Lie algebra endowed with a semisimple abelian complex structure $J$. The \emph{mean curvature vector} $JU$  of $(\f g, B^{1,1}_s)$ is defined implicitly by:
$$
{\rm tr}(\ad_X)=B_s^{1,1}(JU, X)\,, \quad X \in \f g\,.
$$
\end{defn}

The terminology is standard, see e.g. \cite{Einstein}. In the next result we collect some properties of the mean curvature vector in our setting. 

\begin{lem}\label{ILGRANDEJU}
Let $\f g$ be a  Lie algebra endowed with a semisimple, abelian complex structure $J$. Then 
$$
JU=2\sum_{i=1}^{\dim_{\C}\f g_+}JX_i\,,
$$ where $\{X_1, \ldots, X_{\dim_{\C}\f g_+}\}$ is the basis of $\f g_+\cap \f g'$  contained in the standard basis $\mathcal B$. Hence, $[JU, U]=2U$ and, more generally, \begin{equation}\label{ad_JU2Id}\ad_{JU}|_{\f g'}=2{\rm Id}_{\f g'}\,.
\end{equation}
Finally, if $\varphi\colon \f g_1\to \f g_2$ is a holomorphic isomorphism of $J$-semisimple Lie algebras with abelian complex structure, then $\phi(JU_{\f g_1})=JU_{\f g_2}.$ 
\end{lem}
\begin{proof}
First of all, we define $JY:=2\sum_{i=1}^{\dim_{\C}\f g_+}JX_i$ and, using the bracket relations of the vectors in the basis $\mathcal B$, we  infer that $\ad_{JY}|_{\f g'}=2{\rm Id}_{\f g'}$. To prove the first claim, we observe that, for any $X\in \f g'$, $B_{s}^{1,1}(JY, X)=0$. Using the fact that $JY\in\f g_+$ and \eqref{penis} we deduce
$$
B_s^{1,1}(JY, JX)=\frac12 {\rm tr}(\ad_{J[JY, X]})={\rm tr}(\ad_{JX})\,,
$$
thereby proving the first three statements. 

The last claim is proved by straightforward computations once we observe that   $\varphi s_{\f g_1 }\varphi^{-1}=s_{\f g_2 }$, which follows from $\varphi((\f g_1) _{\pm})=(\f g_2)_{\pm}$. This property can be deduced, for instance, using the basis \eqref{baseincredibile} and the fact that $\varphi$ is a Lie algebra isomorphism.
\end{proof}

Next, we introduce an operator which will  help us to obtain the desired Levi--Malcev decomposition.
\begin{defn}\label{defn_Cas}
Let $\f g$ be a  Lie algebra endowed with a semisimple abelian complex structure $J$. The \emph{$J$-adapted Casimir operator}  is defined as 
$$
\Omega^{1,1}_{\f g}:=\frac12(\ad_{JU}-J\ad_{U})\in{\rm End}(\f g, J)\,.
$$
\end{defn}
\begin{rmk}\label{Casimiro}
We now justify the name of $\Omega_{\f g}^{1,1}.$ To this end, we consider the standard basis $\mathcal B $ defined in \eqref{baseincredibile} and 
 we notice that, for any $i=1, \ldots, p$ and $j=1, \ldots, q$, 
 $$
B^{1,1}(X_i,X_i)=\frac12\,,\quad B^{1,1}(X_{p+j}, X_{p+j})=-B^{1,1}(Y_{p+j}, Y_{p+j})=1\,.
 $$
 Hence, the dual basis of $\mathcal B$ with respect to $B^{1, 1}$, see \cite[Chapter 5, p.238]{K02},  is given by
$$
\tilde{\mathcal B}:=\{2X_1, 2JX_1, \ldots,2X_{p}, 2JX_p, X_{p+1}, JX_{p+1}, \ldots,  X_{p+q}, JX_{p+q}, -Y_{p+1},-JY_{p+1}, \ldots, -Y_{p+q}, -JY_{p+q}\}\,.
$$
By \eqref{penis} and the explicit expressions for the Lie brackets of elements in $\mathcal B$, the Casimir operator $\Omega_{\f g}$ takes the form
$$
\begin{aligned}
\Omega_{\f g}=&\, \sum_{i=1}^pB^{1,1}(X_i, X_i)(\ad^2_{2X_i}+\ad_{2JX_i}^2)+ \sum_{j=1}^{q}B^{1,1}(X_{p+j}, X_{p+j})(\ad_{X_{p+j}}^2+\ad_{JX_{p+j}}^2-\ad_{Y_{p+j}}^2-\ad_{JY_{p+j}}^2)\\
=&\, 2\left(\sum_{i=1}^p\ad_{JX_i}+\sum_{j=1}^q\ad_{JX_{p+j}}\right)=\ad_{JU}\,.
\end{aligned}
$$
Hence, 
$$
\Omega_{\f g}^{1,1}=\frac12(\Omega_{\f g}-J\Omega_{\f g}J)\,.
$$
 
Moreover,  it turns out that $\Omega^{1,1}_\g$ coincides with the identity operator on $\g$.  Indeed, by \eqref{ad_JU2Id}, we have that $J\ad_U|_{J\f g' }=-2{\rm Id}_{J\f g'}$. Hence,  
$ \Omega_{\f g}^{1,1}={\rm Id}_{\f g}$, since $\f g$ is $J$-perfect.
\end{rmk}
 
Before presenting the proof of Theorem \ref{mainabelian}, we give the following  definition and prove a general lemma that will be used later.

\begin{defn}
Let $\g$ be a Lie algebra endowed with a complex structure $J$. We call a $J$-invariant subalgebra  $\h$  of $\g$ a $J$-\emph{adapted Levi subalgebra} of $\g$  if
\[
\g=\h \ltimes {\rm Rad}(\g)\,.
\]
\end{defn}

From Lemma \ref{lemss} it is evident that a $J$-adapted Levi subalgebra is necessarily $J$-semisimple as it is holomorphically isomorphic to $\g / {\rm Rad}(\g)$.

\begin{lem}\label{Lemmino}
Let $\f g$ be a Lie algebra endowed with a complex structure $J$. Let $\f k_1, \f k_2\subseteq \f g$ be two $J$-invariant ideals such that $\f k_1\subseteq \f k_2$ such that $\f k_2'+J\f k_2'\subseteq \f k_1$  and $\g'_J \subseteq {\rm Rad}(\f k_2)$. Then
\[
{\rm Rad}(\f k_1)={\rm Rad}(\f k_2)\cap \f k_1\,, \qquad \f k_2=\f k_1+ {\rm Rad}(\f k_2)\,.
\]
\end{lem}
\begin{proof}
From the inclusion $\f k_1\subseteq \f k_2$, we deduce that  ${\rm Rad}(\f k_2)\cap\f k_1$ is a $J$-solvable ideal of $\f k_1$. Hence, ${\rm Rad}(\f k_2)\cap \f k_1\subseteq {\rm Rad}(\f k_1)$.  On the other hand, 
 $$
 \frac{\f k_1}{{\rm Rad}(\f k_2)\cap \f k_1}\simeq\frac{\f k_1 + {\rm Rad}(\f k_2)}{{\rm Rad}(\f k_2)}\,,
 $$
 which is a $J$-invariant ideal of $ \f k_2/{\rm Rad}(\f k_2)$ and thus $J$-semisimple, due to  the assumption $\g'_J \subseteq {\rm Rad(\f k_2)}$ and Theorem \ref{Semisimple}. Therefore, 
 $$
 0={\rm Rad}\left(\frac{\f k_1}{{\rm Rad}(\f k_2)\cap \f k_1}\right)=\frac{{\rm Rad}(\f k_1)}{{\rm Rad}(\f k_2)\cap  \f k_1},
 $$
 concluding that ${\rm Rad}(\f k_1 )={\rm Rad}(\f k_2)\cap \f k_1$.  As for the second identity, note that the quotient
\[
\frac{\frac{\f k_2}{{\rm Rad}(\f k_2)}}{\frac{\f k_1 + {\rm Rad}(\f k_2)}{{\rm Rad}(\f k_2)}} \simeq \frac{\f k_2}{\f k_1+{\rm Rad}(\f k_2)}
\]is $J$-semisimple and abelian, since $\f k_2'+J\f k_2'\subseteq \f k_1$. This can only happen if $\f k_1 + {\rm Rad}(\f k_2)=\f k_2$.
\end{proof}

\begin{rmk}
We notice that if $J$ is abelian  then $\g'_J \subseteq {\rm Rad}(\g)$. Indeed, $\g'_J \subseteq \f z(\g'+J\g')$ and $ \f z(\g'+J\g')$ is a $J$-invariant abelian ideal of $\g$, and as such, it is contained in ${\rm Rad}(\g)$.
\end{rmk}

We can now establish the $J$-adapted Levi-Malcev decomposition in the particular case when $\g'_J=0$ which is later extended to arbitrary abelian complex structures.
\begin{thm}\label{propsemidir}
Let $\f g=\f a\rtimes J\f a$ be a  Lie algebra with an abelian complex structure $J$ such that $\f a$ is an abelian ideal and $\f g_J'=0$. Then there exists a unique $J$-invariant subalgebra $\f h$ such that 
$$
\f g= \f h\ltimes\f n_J\,.
$$
\end{thm}
\begin{proof}
Using \cite[Theorem 3.5]{ABD12}, we know that 
$
\f g\simeq \f{aff}(A)
$
for some finite-dimensional,  associative and commutative $\R$-algebra $A$. Using the Wedderburn's Principal Theorem, see \cite[Corollary, Section 11.6]{Pie82},  we can write 
$$
A=J\oplus S\, ,
$$
where $J$ is the Jacobson radical and $S$ is a semisimple subalgebra of $A$ such that $S\simeq A/J$, see \cite[Corollary b, Section 10.7]{Pie82}. On the other hand, using the finite dimensionality of $A$ as an $\R$-algebra, we have that $J=N$, where $N$ is the nilradical of $A$, i.e. the maximal ideal consisting of nilpotent elements. Therefore, 
$$
A=N \oplus S\,.
$$
Once we have this, we can use Lemma \ref{lemmaffini}, part \emph{(3)}, to infer  
$$
\f g\simeq \f{aff}(A)=\f{aff}(N\oplus S)= \f{aff}(N)\ltimes \f{aff}(S)\,.
$$
Since $S$ is semisimple,  $\f{aff}(S)$ is the direct sum of $J$-simple ideals by Lemma \ref{lemmaffini}, part \emph{(3)}, and hence $J$-semisimple by Theorem \ref{Semisimple}. On the other hand, using \cite[Lemma 2.3]{AO15}, we have that $\f{aff}(N)$ is a $J$-invariant nilpotent ideal, consequently $\f{aff}(N)\subseteq \f n_J$. Since $\f{aff}(S)\simeq \f{aff}(A/N)$ is $J$-semisimple, we conclude $\f{aff}(N)=\f n_J$. 

We now prove the uniqueness part of the statement. Let us assume that there exist two  $J$-adapted Levi subalgebras $\f h_1, \f h_2$ of $\f g$. Using  the fact that $\f h_i$ is $J$-perfect and Lemma \ref{lemmaffini}, part \emph{(2)}, we know that there exist a subalgebra $S_i\subseteq A$ such that $\f h_i=\f{aff}(S_i)$, for any $i=1,2$.  Thanks to Lemma \ref{lemmaffini}, part \emph{(3)}, and since $\f{aff}(N)=\f n_J$, we deduce that $A=S_1\oplus N=S_2\oplus N$. Now, we can apply \cite[Corollary, Section 11.6]{Pie82} to conclude that there exists $n\in N$ such that $S_1=(1-n)S_2(1-n)^{-1}$. We obtain $S_1=S_2$ by commutativity of $A$. The claim now follows. 
\end{proof}

With this established, we can now prove Theorem \ref{mainabelian}.

\begin{thm}\label{Leviabelian}
Let $\g$ be a Lie algebra equipped with an abelian complex structure $J$. Then there exists a unique $J$-invariant subalgebra $\f h$ such that 
\[
\g = \h \ltimes {\rm Rad}(\f g)\,.
\]
\end{thm}
\begin{proof}
We prove the existence part of the theorem in a stepwise fashion by successively reducing to simpler cases.

{\bf Step 1}. We show that we may assume that $\g$ has no $J$-invariant ideals properly contained in ${\rm Rad}(\f g)$. To prove this we work inductively on $\dim_\C(\g)$. If $\g$ is $1$-dimensional, it can only be either $\mathfrak{aff}(\R)$ or $\C$ and in both cases the theorem is clearly true. Thus, let $\f k \subset {\rm Rad}(\f g)$ be a $J$-invariant ideal in $\g$. Since $\mathrm{Rad}(\g/\f k)={\rm Rad}(\f g)/\f k$, by the inductive assumption we have $\g/\f k = \f s \ltimes {\rm Rad}(\f g)/\f k$,  for some $J$-semisimple subalgebra $\f s$. Consider now the canonical projection onto the quotient $\pi \colon \g \to \g/\f k$, then $\hat{\f s}:= \pi^{-1}(\f s)$ is a proper $J$-invariant subalgebra of $\g$ and so, by the inductive assumption again, it can be written as $\hat{\f s}= \h \ltimes \mathrm{Rad}(\hat{\f s})$ for a $J$-semisimple subalgebra $\h$. Since $\hat{\f s}/\f k=\pi(\pi^{-1}(\f s))= \f s$ is $J$-semisimple, it follows that $\mathrm{Rad}(\hat{\f s})\subseteq  \f k$, but $\f k\subseteq  {\rm Rad}(\f g)$ is $J$-solvable and thus we must have $\f k= \mathrm{Rad}(\hat{\f s})$. It then follows that $\g= \h \ltimes  {\rm Rad}(\f g)$.

{\bf Step 2}. We reduce here to the case where $\g=\g'+J\g'$. 
Using Lemma \ref{Lemmino} with $\f k_1=\g'+J\g'$ and $\f k_2=\g$,  we have $\g=\g'+J\g'+{\rm Rad}(\g)$ and ${\rm Rad}(\f g'+J\f g')={\rm Rad}(\f g)\cap (\f g'+J\f g')$. The latter is a $J$-invariant ideal of $\f g$ contained in ${\rm Rad}(\f g)$. Using Step 1, we may assume either ${\rm Rad}(\f g'+J\f g')=0$ or ${\rm Rad}(\f g'+J\f g')={\rm Rad}(\f g)$. In the first case, the Levi--Malcev decomposition follows from $\g=(\g'+J\g')\oplus{\rm Rad}(\g)$. In the second case we have ${\rm Rad}(\f g)\subseteq \f g'+J\f g'$, and hence $\g=\g'+J\g'+{\rm Rad}(\g)\subseteq \f g'+J\f g'$, as claimed.  

{\bf Step 3}. We show that it may be assumed that $\f z(\g)=0$. We suppose $\f z(\g)\neq 0$ and provide the Levi--Malcev decomposition for this case. Since $J$ is abelian,  $\f z(\g)$ is $J$-invariant and thus contained in the radical of $J$. Therefore, again by Step 1, we can suppose $\f z(\g)= {\rm Rad}(\f g)$. In these assumptions, the adjoint representation $\ad \colon \g \to \mathrm{End}(\g)$ descends to a well-defined representation of the quotient 
\[
\ad\colon \frac{\g}{{\rm Rad}(\f g)} \to \mathrm{End}(\g)\,.
\]
We consider the operator
\[
\Phi(U):=\frac12 \left(\ad_{JU}-J\ad_{U}\right) \in \mathrm{End}(\g)\,, 
\] where $JU$ is the mean curvature vector of $\f g/{\rm Rad}(\f g)$. 
It is easy to prove that $\Phi(U) \in \mathrm{End}(\g, J)$, namely $[\Phi(U),J]=0$.  Moreover, for every $X\in \g/{\rm Rad}(\f g)$,  
we have 
$$
\begin{aligned}
\relax[\Phi(U), \ad_X]=&\, \frac12(\ad_{JU}\ad_X-J\ad_U\ad_X-\ad_X\ad_{JU}+\ad_XJ\ad_U)\\=&\,\frac12(\ad_{JU}\ad_X-\ad_X\ad_{JU}-\ad_{JX}\ad_U+\ad_U\ad_{JX})=\frac12(\ad_{[JU, X]+[U, JX]})=0\,,
\end{aligned}
$$
where we used the fact that, since $\f g$ is $2$-step and $U\in (\f g/{\rm Rad}(\f g))'$, $\ad_U\ad_X=\ad_U\ad_{JX}=0.$ 
As a consequence, $\ker\Phi(U)$ and $\Im \Phi(U)$ are $J$-invariant ideals of $\g$. We now consider the operator
\[
\Omega^{1,1}_{\g/ {\rm Rad}(\f g)}:=\frac12 \left(\ad^{\g/{\rm Rad}(\f g)}_{JU}-J\ad^{\g/{\rm Rad}(\f g)}_{U}\right) \in \mathrm{End}\left(\frac{\g}{{\rm Rad}(\f g)}, J\right)\,,
\]
defined so that $\Omega^{1,1}_{\g / {\rm Rad}(\f g)} \circ \pi = \pi \circ \Phi(U)$, where $\pi \colon \g \to \g / {\rm Rad}(\f g)$ is the canonical projection onto the quotient. By Remark \ref{Casimiro} $\Omega^{1,1}_{\g / {\rm Rad}(\f g)}=\mathrm{Id}_{\g / {\rm Rad}(\f g)}$, whence $0=\ker \Omega^{1,1}_{\g / {\rm Rad}(\f g)}=\ker\Phi(U)/{\rm Rad}(\f g)$, i.e. $\ker\Phi(U)={\rm Rad}(\f g)$. We then conclude that $\g=\Im \Phi(U)\oplus {\rm Rad}(\f g)$, completing the proof of this step.

{\bf Step 4}. Finally, we conclude the existence part of the theorem. Thanks to the previous steps we may assume $\g= \g' + J \g'$ and $ \f z(\g)=0$. Since $\f z(\g)=0$, in particular $\g'_J=0$, which puts us in the position to apply Theorem  \ref{propsemidir} to deduce the desired decomposition of $\g$.

In order to prove the uniqueness part of  the theorem, we proceed by induction on $\dim_{\C}\f g$. If $\dim_{\C }\f g=1$, then either $\f g=\C$ or $\f g=\f {aff}(\R)$ and the statement is clearly true. Assume the assertion is true for any Lie algebra whose complex dimension is strictly less than $\dim_{\C}\f g$.
Let $\f h_1, \f h_2$ be two $J$-adapted Levi subalgebras of $\f g.$ Assume first $\f g\ne \f g' + J\f g' $. Then, $\f h_1, \f h_2\subseteq \f g'+J\f g' $, because they are both $J$-perfect. Furthermore, thanks to Lemma \ref{Lemmino}, we deduce that $\h_1, \h_2$ are $J$-adapted Levi subalgebras of $\g'+J\g'$. We then apply the inductive hypothesis and conclude that they are equal. 
Then, it remains to treat the case when $\f g=\f g'+J\f g'$. If $\f z( \f g)=0$, then $\f g_J'=0$ and we may conclude applying Theorem \ref{propsemidir}. Hence, we may assume that $\f z(\f g)\ne 0 $. We distinguish two cases. Let us assume first that $\f z(\f g)\ne {\rm Rad}(\f g)$. Consider the canonical projection onto the quotient $\pi\colon\f g \to\f g/\f z(\f g)$ and observe that $\pi(\f h_i)$, $i=1,2$, are $J$-adapted Levi subalgebras of $\f g/\f z(\f g)$. By the inductive hypothesis, $\pi(\f h_1)=\pi(\f h_2)$. This allows to infer that $\f h_1\subseteq \f h_2+\f z(\f g):=\f s$. Of course, ${\rm Rad}(\f s )=\f z(\f g)$, hence $\f h_1, \f h_2$ are $J$-adapted Levi subalgebras of $\f s$ and $\dim_{\C}\f s<\dim_{\C}\f g$, since $\f z(\f g)\ne {\rm Rad}(\f g)$. We  conclude that $\f h_1=\f h_2$ by the inductive hypothesis. It remains to understand the case $\f z(\f g)={\rm Rad}(\f g)$. 
In terms of the decomposition $\f g=\f h_1\ltimes \f z(\f g)$, we  consider the projections $p_{\f h_1}$ and $p_{\f z(\f g)}$ on $\f h_1$ and $\f z(\f g)$, respectively. First of all, we observe that, since $J\f h_1=\f h_1$ and $J\f z(\f g)=\f z(\f g)$, $[p_{\f h_1}, J]=[p_{\f z(\f g)}, J]=0$. Hence, for any $Z\in \f h_2$, we can write 
$$
Z=p_{\f h_1}(Z)+ p_{\f z(\f g)}(Z)\,.
$$
It follows that,  for any $Z_1, Z_2\in \f h_2$,  
$$
p_{\f h_1}([Z_1, Z_2])+ p_{\f z(\f g)}([Z_1, Z_2])=[Z_1, Z_2]=[p_{\f h_1}(Z_1), p_{\f h_1}(Z_2)]\,,
$$
showing that $p_{\f h_1}$ is a holomorphic Lie algebra homomorphism and that $p_{\f z(\f g)}|_{\f h_2'}=0$. Since $\f h_2=\f h_2'\oplus J\f h_2'$ and $[p_{\f z(\f g)}, J]=0$ it follows that $p_{\f z(\f g)}|_{\f h_2}=0$.  Thus, for any $Z\in\f h_2$,  $Z=p_{\f h_1}(Z)\in\f h_1$ and so $\f h_2\subseteq \f h_1$ forcing $\f h_2=\f h_1$, since they have the same dimension, thereby completing the proof. 
\end{proof}

A first consequence of Theorem \ref{Leviabelian} is the following result.

\begin{cor}\label{lemspettroB}
Let $\f g$ be a Lie algebra endowed with an abelian complex structure $J$. Then the adjoint operators associated with elements of the standard basis $\mathcal B$ of the $J$-adapted Levi subalgebra $\f h$ defined in \eqref{baseincredibile}, $\{\ad_{JX_1}, \ldots, \ad_{JX_{p+q}}, \ad_{JY_{p+1}}, \ldots, \ad_{JY_{p+q}}\}$, are simultaneously diagonalisable, with 
$$
{\rm Spec}(\ad_{JX_i})\subseteq\{0,1\}\,, \quad  {\rm Spec}(\ad_{JY_{p+j}})\subseteq\{0, \sqrt{-1}, -\sqrt{-1}\}\,, \quad i=1, \ldots, p+q\,,\quad j=1, \ldots, q\,.
$$
In particular, for any $X\in \f h'$, $\ad_{JX}$ is diagonalisable. 
\end{cor}
\begin{proof}
First of all, for any $i=1, \ldots, p+q$ and $j=1, \ldots, q$, using \eqref{penis} multiple times, we deduce 
$$
\begin{aligned}
\ad_{JX_i}^2&=\ad_{J[JX_i, X_i]}=\ad_{JX_i}\,,\\
\ad_{JY_{p+j}}^3&=\ad_{J[JY_{p+j},Y_{p+j}]}\ad_{JY_{p+j}}=-\ad_{JX_{p+j}}\ad_{JY_{p+j}}=-\ad_{J[JX_{p+j}, Y_{p+j}]}=-\ad_{JY_{p+j}}\,,
\end{aligned}
$$
which immediately imply that $\ad_{JX_i}$ and $\ad_{JY_{p+j}}$ are diagonalisable with ${\rm Spec}(\ad_{JX_i})\subseteq\{0,1\}$ and ${\rm Spec}(\ad_{JY_{p+j}})\subseteq\{0, \sqrt{-1}, -\sqrt{-1}\}$.
Finally, since $J$ is abelian, $\{\ad_{JX_1}, \ldots, \ad_{JX_{p+q}}, \ad_{JY_{p+1}}, \ldots, \ad_{JY_{p+q}}\}$ is a family of commuting and diagonalisable endomorphisms of $\f g$ which are, hence, simultaneously diagonalisable.  The last statement is now clear.
\end{proof}

As a consequence of Corollary \ref{lemspettroB}, we obtain the following.

\begin{cor}\label{abelianpostLevi}
Let $\g$ be a Lie algebra endowed with an abelian complex structure $J$.
\begin{enumerate}
    \item If $\g'+J\f g'$ is unimodular, then $\g$ is $J$-solvable.
    \item If $\g$ is of real type, then the $J$-adapted Levi subalgebra is holomorphically isomorphic to $\f{aff}(\R)^p$, $p\geq 0$.
\end{enumerate}
\end{cor}
\begin{proof}
Using Corollary  \ref{lemspettroB}, we observe that unimodularity of $\g'+J\f g'$ would not allow the $J$-adapted Levi subalgebra $\f h$ to be  non-zero, proving \emph{(1)}. As for \emph{(2)}, we notice that, again by Corollary  \ref{lemspettroB}, if $\f g$  is of real type then there cannot be any ideal of the $J$-adapted Levi subalgebra in the decomposition given by Corollary \ref{corsymmetricdecom} which is isomorphic to $\f{aff}(\C)$. 
\end{proof}

\section{Levi--Malcev decomposition}\label{seclevi}
We begin by emphasising that proving a Levi--Malcev decomposition in the general framework of $2$-step solvable Lie algebras with complex structure is actually an impossible task, as shown by the following example.

\begin{es}\label{noLevi}
Let us consider the almost abelian Lie algebra $\f g$ defined by:
$$
[e_1, e_2]=e_2+e_3+e_4\,,\quad [e_1, e_3]=e_3\,,\quad [e_1, e_4]=e_4\,.
$$
In the notations of \cite{SW14}, such Lie algebra is isomorphic to $\f s_{4,3}$ with $a=b=1$. We endow $\g$ with the complex structure $Je_1=e_2$ and $Je_3=e_4$. As it is evident, $\f g_J'={\rm Rad}(\f g)=\langle e_3,e_4\rangle$ and $\f g/\f g_J'\simeq \f{aff }(\R)$. It is fairly easy to see that there is no $J$-adapted Levi subalgebra of $\f g $. Moreover, we also  note that  $\f g \ltimes \langle e_5, e_6\rangle$ with $Je_5=e_6$ and 
$$
[e_1, e_5]=-\frac32e_5\,, \quad [e_1, e_6]=-\frac32e_6\,,
$$ is a unimodular, $2$-step solvable Lie algebra with no $J$-adapted Levi subalgebra.
\end{es}

On the other hand, the previous example gives us an indication of the obstruction to take into account if one wishes to prove a general $J$-adapted Levi--Malcev decomposition. We now work towards the formalisation of such a condition.

\begin{defn}
Let $\f g $ be a $2$-step solvable Lie algebra with a complex structure $J$. Consider the subspace $\mathcal{O}(\g'_J)$ of elements in $\g$ that act holomorphically on $\g'_J$, i.e.
$$
\mathcal O (\f g_J'):=\{X\in \f g \, |\, [\ad_X, J]|_{\f g_J'}=0\}\,.
$$
Notice that $\g'+J\g' \subseteq \mathcal{O}(\g'_J)$ by Lemma \ref{Lem2stepideal}, and so $\mathcal{O}(\g'_J)$ is an ideal in $\g$. By integrability it is also $J$-invariant. Furthermore,  $\g'_J$ is a $J$-invariant ideal in $\mathcal{O}(\g'_J)$.
\end{defn}
  Since $\f g_J'$ is abelian, we can define the representation
$$
\tau\colon \frac{\mathcal{O}(\g'_J)}{\f g_J'}\to {\rm End}(\f g_J')\,, \quad \tau(X+\f g_J'):=\ad_X|_{\f g_J'}\,.
$$ 
On the other hand, using Theorem \ref{Leviabelian} we can consider the $J$-adapted Levi subalgebra $\f h$ of $\mathcal{O}(\g'_J)/\g'_J$. Within $\h$, we take into account the subalgebra $\f h_{JU}:=\langle U, JU\rangle$, where $JU$ is the mean curvature vector of $\h$, see Definition \ref{defn_U}.
Then, we define
$$
\tau_{JU}:=\tau|_{\f h_{JU}}\,,
$$
which turns out to be a representation of $\f h_{JU}$ on $\f g_J'$.  Observe that $\tau_{JU}(U)=0$, because $U\in (\mathcal{O}(\g'_J)/ \f g_J')'$.
The obstruction we are interested in is encoded in the cohomology group arising from the representation $\tau_{JU}$, see Definition \ref{cohomologiagen}. The $J$-adapted Levi subalgebra $\h$ is unique, hence, the following is well-defined.
\begin{defn}\label{cohomolob}
Let $\g$ be a $2$-step solvable Lie algebra with a complex structure $J$. Then we define the cohomology group
$$
H^2(\g, JU):= H^{2}(\f h_{JU}, \tau_{JU})\,.
$$
\end{defn}

With a slight abuse of notation, we will often write $JU$ to denote also a representative in $\mathcal{O}(\g'_J)$ of the equivalence class of the mean curvature vector of $\f h$.

Our interest in the just defined cohomology group stems from the following characterisation.

\begin{lem}\label{obstruction}
Let $\g$ be a $2$-step solvable Lie algebra with a complex structure $J$. Then $H^2(\g, JU)=0$ if and only if $2\notin {\rm Spec}(\ad_{JU}|_{\f g_J'})$.
\end{lem}
\begin{proof}
Of course $\ker d_{\tau_{JU}}\cap V^2(\h_{JU}, \tau_{JU})=V^2(\h_{JU}, \tau_{JU})$, since $\dim_{\R}\f h_{JU}=2$. On the other hand, 
$$
(d_{\tau_{JU}}L)(JU, U)=\ad_{JU}L(U)-2L(U)=(\ad_{JU}-2{\rm Id})L(U)\,,
$$
using the fact that $[JU, U]=2U$ and that $\tau_{JU}(U)=0$. Now, any $\alpha \in V^2(\h_{JU}, \tau_{JU}) $ is completely determined by $\alpha(JU, U)\in \f g_J'$. Hence, the condition $H^2(\g, JU)=0$ is equivalent to the surjectivity of $(\ad_{JU}-2{\rm Id})|_{\f g_J'}$ which, in turn, is equivalent to $2\notin {\rm Spec}(\ad_{JU}|_{\f g_J'})$.
\end{proof}

Before stating and proving a preliminary version of Theorem \ref{Main1}, we show that the vanishing of $H^2(\g,JU)$ is inherited by suitable $J$-invariant subalgebras.

\begin{lem}\label{unnosomicasai}
Let $\f g$ be a $2$-step solvable Lie algebra endowed with a complex structure $J$ such that $H^{2}(\f g, JU)=0$. Then, for any $J$-invariant  subalgebra $\hat {\f s} \subseteq \mathcal O(\f g_J')$ satisfying ${\rm Rad}(\hat{\f s})\subseteq {\rm Rad}(\mathcal O(\f g_J'))$ and $\hat {\f s}/{\rm Rad}(\hat {\f s})\simeq \mathcal O(\f g_J')/{\rm Rad}(\mathcal O(\f g_J'))$, we have $H^2(\hat{\f s}, JU_{\hat{\f s}})=0.$
\end{lem}
\begin{proof}
Let $\hat {\f s}$ be a subalgebra as in the statement. It is apparent that $\hat {\f s}_J'\subseteq \f g_J'$ and, since $\hat {\f s}\subseteq \mathcal O(\f g_J')$, we have $\hat {\f s}=\mathcal O_{\hat {\f s}}(\hat {\f s}_J'):=\{X\in \hat {\f s}\, \, |\, \, [\ad_X, J]|_{\f s_J'}=0\}$. Furthermore, since $\hat {\f s}_J'\subseteq \f g_J'$,  the inclusion $\hat {\f s}\to \f g$ induces a Lie algebra homomorphism 
$$
\varphi\colon \frac{\hat {\f s}}{\hat {\f s}_J'}\to \frac{\mathcal O(\f g_J')}{\f g_J'}\,, \quad \varphi(X+\hat {\f s}_J')=X+\f g_J'\,.
$$
Now, by Theorem \ref{Leviabelian}, we have that 
$$
\frac{\hat {\f s}}{\hat {\f s}_J'}=\f h_{\hat {\f s}}\ltimes \frac{{\rm Rad}(\hat {\f s})}{\hat {\f s}_J'}\,,\quad \frac{\mathcal O(\f g_J')}{\f g_J'}=\f h_{\f g}\ltimes\frac{{\rm Rad}(\mathcal O (\f g_J'))}{\f g_J'}\,. 
$$ 
Then, since $\ker(\varphi|_{\f h_{\hat {\f s}}})=\ker \varphi\cap \f h_{\hat {\f s}}=\frac{\f g_J'}{\hat {\f s}_J'}\cap \f h_{\hat {\f s}}=0$, we see that $\varphi|_{\f h_{\hat {\f s}}}$ is a holomorphic Lie algebra isomorphism onto ${\rm Im}(\varphi|_{{\f h}_{\hat {\f s}}})$. In particular, this implies that ${\rm Im}(\varphi|_{\f h_{\hat {\f s}}})$ is a $J$-semisimple Lie subalgebra of $\mathcal O(\f g_J')/\f g_J'$, and hence ${\rm Im}(\varphi|_{\f h_{\hat {\f s}}})\cap \frac{{\rm Rad}(\mathcal O (\f g_J'))}{\f g_J'}=0$. On the other hand, 
$$
\f h_{\hat {\f s}}\simeq \frac{\hat {\f s}}{{\rm Rad}(\hat {\f s})}\simeq \frac{\mathcal O(\f g_J')}{{\rm Rad}(\mathcal O(\f g_J'))},
$$ which gives that ${\rm Im}(\varphi|_{\f h_{\hat {\f s}}})$ is a $J$-adapted Levi subalgebra of  $\mathcal O(\f g_J')/\f g_J'$. Using Theorem \ref{Leviabelian}, we conclude that $\f h_{\f g}= {\rm Im}(\varphi|_{\f h_{\hat {\f s}}})$. Then, from the last statement in Lemma \ref{ILGRANDEJU} we deduce 
$$
JU_{\hat {\f s}}+ \f g_J'=\varphi(JU_{\hat {\f s}}+ \hat {\f s}_J')=JU_{\f g}+ \f g_J'. 
$$ The claim now follows, since $\ad_{JU_{\hat {\f s}}}|_{\hat {\f s}_J'}=\ad_{JU_{\f g}}|_{\hat {\f s}_J'}$. 
\end{proof}

\begin{cor}\label{corobgprimo}
Let $\f g$ be a $2$-step solvable Lie algebra endowed with a complex structure $J$ such that $H^2(\f g, JU)=0$. Then $H^2(\f g'+J\f g', JU_{\f g'+J\f g'})=0.$
\end{cor}
\begin{proof}
We have that $\f g'+J\f g'$ satisfies the hypotheses of Lemma \ref{unnosomicasai}, thanks to Corollary \ref{Lemmino}. The conclusion follows.  
\end{proof}

The next result proves the desired Levi-Malcev decomposition when the Lie algebra $\f g$ acts holomorphically on $\f g_J'$ and the cohomology group defined above vanishes.

\begin{thm}\label{levihol}
Let $\f g$ be a $2$-step solvable Lie algebra endowed with a complex structure $J$ such that $\f g=\mathcal O(\g'_J)$, and suppose that $H^2(\g, JU)=0$. Then there exists a $J$-semisimple subalgebra $\f h$  of $\f g$ such that $$
\f g=\f h\ltimes {\rm Rad}(\f g)\,. 
$$
\end{thm}
\begin{proof}
 First of all, we note that in our hypotheses $\f g_J'$ is a $J$-invariant ideal of $\f g$. With this in mind,  we divide the proof into steps.

{\bf Step 1}. We show that we may assume ${\rm Rad}(\f g)=\f g_J'\ne 0 $ and that $\g$ has no $J$-invariant ideals properly contained in ${\rm Rad}(\f g)$. 
First, note that we may suppose $\f g_J'\ne 0$, otherwise we can simply apply Theorem \ref{Leviabelian} and conclude. To prove that we may reduce ourselves to the case ${\rm Rad}(\f g)=\f g_J'$,  we work inductively on $\dim_\C(\g)$. If $\g$ is $1$-dimensional it can only be either $\mathfrak{aff}(\R)$ or $\C$ and in both cases the theorem is clearly true. 
Let us assume that $\f g_J'\subset{\rm Rad}(\f g)$. Then,  since $\f g/ \f g_J'$ has an abelian  complex structure, we  can apply Theorem \ref{Leviabelian} and obtain 
$$
\frac{\f g}{\f g_J'}=\f s\ltimes \frac{{\rm Rad}(\f g)}{\f g_J'}\,, 
$$
where $\f s $ is the  $J$-adapted Levi subalgebra of $\f g /\f g_J'$. Now consider $\hat{\f s}:=\pi^{-1}(\f s)$, which is a proper $J$-invariant subalgebra of $\f g$. Since $\hat{\f s}/\f g_J'=\f s$ is $J$-semisimple, we can infer, using the same strategy as in Step 1 of Theorem \ref{Leviabelian}, that   $\f g_J'= {\rm Rad}(\hat{\f s})$.  We may then apply Lemma \ref{unnosomicasai} to infer that $H^2(\hat{\f s},JU_{\hat{\f s}})=0$. 
Hence, by the inductive hypothesis $\hat{\f s}$ admits a Levi--Malcev decomposition, and therefore so does $\g$. Consequently, we can assume $\f g_J'={\rm Rad}(\f g).$ 

Finally, we prove that $\f g$ has no $J$-invariant ideals properly contained in ${\rm Rad}(\f g)$, once again by induction on $\dim_\C\f g$. If $\dim_\C\f g=1$, the claim is trivial.  
Let $\f k \subset \f g_J'$ be a $J$-invariant ideal of $\f g$. Note that $(\f g/\f k)_J'=\frac{\f g_J'}{\f k}$. Furthermore $\frac{\f g/\f k}{\f g_J'/\f k}\simeq \f g/\f g_J'$ via $\varphi((X+\f k)+ \frac{\f g_J'}{\f k})=X+\f g_J'$. On the other hand, using again  Lemma \ref{ILGRANDEJU}, we have $JU_{\f g/\f k}+ \f g'_J=\varphi((JU_{\f g/\f k }+\f k)+\frac{\g'_J}{\f k})=JU_{\f g}+ \f g_J'.$ Denote by $\pi^{\f k }\colon \f g \to\f g/\f k $ the quotient map, then
 $$
\ad_{JU_{\f g/\f k}+ \f k}^{\f g/\f k }|_{\f g_J'/\f k}=\pi^{\f k}\ad_{JU_{\f g/\f k}}^{\f g}|_{\f g_J'}=\pi^{\f k}\ad_{JU_{\f g}}^{\f g}|_{\f g_J'}\,.
 $$
Hence, since $2\notin{\rm Spec}(\ad_{JU_{\f g}}^{\f g}|_{\f g_J'})$, the same  holds for $\ad_{JU_{\f g/\f k}+ \f k}^{\f g/\f k }|_{\f g_J'/\f k}$. Furthermore, since $\f g$ acts holomorphically on $\f g_J'$, then $\f g/\f k$ acts holomorphically on $(\f g/\f k)_J'$. Therefore, applying the inductive hypothesis we obtain a Levi--Malcev decomposition of $\f g/ \f k$. The same argument in Step 1 of Theorem \ref{Leviabelian}, together with an argument analogous to that used in the first part of this proof, guarantees the existence of a $J$-adapted Levi subalgebra for $\g$.

{\bf Step 2}. 
First of all, thanks to Step 1, we may assume  $0\neq \g'_J={\rm Rad}(\f g)$ and $\f g$ has no $J$-invariant ideals properly contained in ${\rm Rad}(\f g)$.  Moreover, $\f g/\f g_J'$ is $J$-semisimple with an abelian complex structure, hence, it is $J$-perfect by Theorem \ref{Semisimple}.
As in Step 2 of Theorem \ref{Leviabelian}, we can apply Lemma \ref{Lemmino} to conclude that $\f g $ is $J$-perfect. In this setting, we consider 
$$
\Phi\colon \f g \to {\rm End}(\f g  )\,,\quad \Phi(X)=\frac12(\ad_{JX}-J\ad_X)\,, \quad X \in\f g\,.
$$
First of all, using integrability of $J$, for any $X\in \g$, we deduce 
\begin{equation}\label{cazzettiinano}
[\Phi(X), J]=0\,, \quad \Phi(JX)=-\Phi(X)J\,.
\end{equation}
By \eqref{adXg'_J} and \eqref{cazzettiinano}, we know that $\f g_J'\subseteq \ker \Phi$ and $J\ker \Phi=\ker \Phi$. 
Furthermore,  as in Lemma \ref{Lem2stepideal}, we observe that $\f g$ acts holomorphically on $\ker \Phi$.  
Therefore, $[\g, \ker \Phi]\subseteq \f g_J'$, giving that $\ker \Phi$ is a $J$-solvable ideal of $\f g $. Thus, $\f g_J'= \ker \Phi$.
In view of this, $\Phi$ descends to the quotient, i.e. 
$$
\Phi\colon \frac{\f g}{ \f g_J'}\to {\rm End}(\f g).
$$
Moreover, we observe that, if $X\in\f g/\f g_J'$ and $Y\in\f g_J'$, then
\begin{equation}\label{giovannibuco}
\Phi(X)\ad_Y=\ad_{\Phi(X)Y}\,.
\end{equation}
Indeed,  using $2$-step solvability and that $X$ acts holomorphically on $\f g_J'=\ker \Phi$, we have  
$$
\begin{aligned}
\Phi(X)\ad_Y=&\, 
\frac12(\ad_{JX}\ad_Y-J\ad_X\ad_Y)=
\frac12(\ad_{[JX, Y]}-\ad_XJ\ad_Y)=\frac12(\ad_{[JX, Y]}-\ad_X\ad_{JY})\\
=&\, \frac12\ad_{[JX, Y]-[X, JY]}=\ad_{\Phi(X)Y}\,,
\end{aligned}
$$
as claimed.
We will only be interested in $\Phi(U)$, where $JU\in J(\f g /\f g_J')'$ is the mean curvature vector of $\f g /\f g_J'$,  introduced in Definition \ref{defn_U}. 
We extend the endomorphism $\Phi(U)$ by $\C$-linearity on $\f g^{\C}:=\f g\otimes \C$. By abuse of notation, this extension will again be denoted $\Phi(U)$. Since $\f g_J'$ is an ideal of $\f g$, and by  \eqref{cazzettiinano}, $[\Phi(U), J]=0$, we deduce
$$
\Phi(U)((\f g_J')^{1,0})\subseteq (\f g_J')^{1,0}\,,\qquad  \Phi(U)((\f g_J')^{0,1})\subseteq (\f g_J')^{0,1}\,.
$$
Consequently, we can find $\lambda \in \C$ such that $\ker \bigl(\left(\Phi(U)-\lambda {\rm Id}\right)|_{(\f g_J')^{1,0}}\bigr)\ne 0$. By simply exploiting $\C$-linearity, from \eqref{giovannibuco}, we have
$$
\left(\Phi(U)-\lambda {\rm Id}\right)\ad^{\f g^{\C}}_Y=\ad^{\f g^{\C}}_{(\Phi(U)-\lambda {\rm Id})Y}\,, \qquad Y \in (\f g_J')^{1,0}\,.
$$
Hence, by the above and the fact that $\f g$ acts holomorphically on $\f g_J'$, we infer that $\ker \bigl(\left(\Phi(U)-\lambda {\rm Id}\right)|_{(\f g_J')^{1,0}}\bigr)$ and $\ker \bigl(\left(\Phi(U)-\bar \lambda {\rm Id}\right)|_{(\f g_J')^{0,1}}\bigr)$ are $J$-invariant ideals of $\f g^{\C}$ contained in $\f g_J'\otimes \C$. We then consider 
$$
\f k:=\ker \left(\left(\Phi(U)-\lambda {\rm Id}\right)|_{(\f g_J')^{1,0}}\right)\oplus \ker \left(\left(\Phi(U)-\bar \lambda {\rm Id}\right)|_{(\f g_J')^{0,1}}\right)\subseteq \f g_J'\otimes \C
$$
which is a non-zero $J$-invariant ideal of $\f g^{\C}$ which is stable under conjugation. Thus, 
$
\f k=(\f k \cap \f g)\otimes \C\,, 
$
where $\f k \cap \f g\subseteq \f g_J'$ is a non-zero $J$-invariant ideal of $\f g.$ As a consequence, using the first step, we may assume  $\f k\cap \f g= \f g_J'$  and, thus, $\f k= \f g_J'\otimes \C$. 
In particular, we conclude that 
$$
\Phi(U)|_{(\f g_J')^{1,0}}=\lambda {\rm Id}_{(\f g_J')^{1,0}}\,, \quad\Phi(U)|_{(\f g_J')^{0,1}}=\bar \lambda {\rm Id}_{(\f g_J')^{0,1}}\,. 
$$
Denote $a:={\rm Re}(\lambda)$ and $b:={\rm Im}(\lambda)$. Then we infer that
\begin{equation}\label{baumiao}
\Phi(U)|_{\f g_J'}=a{\rm Id}_{\f g_J'}+ b J|_{\f g_J'}\,.
\end{equation}
Now, since $H^2(\g, JU)=0$ and by Lemma \ref{obstruction}, we deduce that $(a, b)\ne (1,0)$. Therefore, from \eqref{baumiao} follows that the map 
$$
\tilde \Psi:=\Phi(U)-a{\rm Id}-b J\colon \f g \to \f g
$$
descends to a well-defined map on the quotient, namely
$$
\Psi\colon \frac{\f g}{ \f g_J'}\to \f g\,.
$$
The first thing to observe is that $\Psi$ is injective. Indeed, recalling Definition \ref{defn_Cas} and  Remark \ref{Casimiro}, 
$$
\pi \tilde \Psi=\pi\Phi(U)-a\pi-bJ\pi=\Omega_{\f g/ \f g_J'}\pi-a\pi- bJ \pi=((1-a){\rm Id}- bJ)\pi\,.
$$
Hence, if $X\in \ker \tilde \Psi $, then $(1-a)\pi(X)=bJ\pi(X)$, which is possible if and only if $\pi(X)=0$, since $(a, b)\ne (1, 0)$, therefore $X\in \f g_J'$. On the other hand, $\Im(\Psi)\cap \f g_J'=0$. Indeed, if $\Psi(\pi (X))\in\f g_J'$, $X\in\f g$,  then again 
$$
0=\pi \Psi(\pi(X))=\pi \tilde \Psi(X)=((1-a)\pi(X)-bJ\pi(X))
$$
implying that $\pi(X)=0$, as claimed. Moreover,  ${\rm Im}\Psi$ is $J$-invariant,  since $[\Psi, J]=0$.

It remains to prove that $\Im \Psi$ is a subalgebra of $\f g$. To do so, we fix $X, Y\in (\f g/\f g_J')'$ and  compute $[\Psi(X),\Psi(Y)]$,  $[\Psi(JX),\Psi(JY)]$ and  $[\Psi(X),\Psi(JY)]$. Let us start  by observing that 
 \[
 \Psi(X)=\frac12[JU,X]-aX-bJX\,, \qquad \Psi(Y)=\frac12[JU,Y]-aY-bJY\,.
 \]
Then,  using $2$-step solvability, Jacobi identity and integrability of $J$,  we get
 \[
 \begin{split}
 [\Psi(X),\Psi(Y)]&=b\left( -\frac12[[JU,X],JY]+a[X,JY]-\frac12[JX,[JU,Y]]+a[JX,Y]+b[JX,JY]\right)\\
 &=b\left( -\frac12[JU,[X,JY]+[JX,Y]]+a[X,JY]+a[JX,Y]+b[JX,JY]\right)\\
&=b\left( \frac12[JU,J[JX,JY]]-aJ[JX,JY]+b[JX,JY]\right)=b\Psi(J[JX,JY])=0\,, 
 \end{split}
 \]  since $J[JX,  JY]\in \f g_J'$. 
 On the other hand, using integrability, we have
 \[
 \Psi(JX)=\frac12J[JU,X]-aJX+bX\,, \qquad \Psi(JY)=\frac12J[JU,Y]-aJY+bY\,,
 \]
 thus
 \[
 \begin{split}
 [\Psi(JX),\Psi(JY)]&=\frac14[J[JU,X],J[JU,Y]]-\frac a2[J[JU,X],JY]+\frac b2[J[JU,X],Y]-\frac a2[JX,J[JU,Y]]\\&\quad +a^2[JX,JY]-ab[JX,Y]+\frac b2[X,J[JU,Y]]-ab[X,JY]\,.
 \end{split}
 \]
 We treat the first term as follows. Using integrability,  $2$-step solvability and Jacobi identity, we have
 \[
\begin{aligned}
\relax
[J[JU, X], J[JU, Y]] &=J([J[JU, X], [JU,Y]]+[[JU, X], J[JU, Y]])\\
&=J\left([JU,[J[JU,X],Y]]+ [JU,[X, J[JU, Y]]]\right)\,.
\end{aligned}
\]
Also, by integrability
\[
[J[JU,X],JY]+[JX,J[JU,Y]]=J[J[JU,X],Y]+J[[JU,X],JY]+J[JX,[JU,Y]]+J[X,J[JU,Y]]\,,
\]
whence
\[
\begin{split}
 [\Psi(JX),\Psi(JY)]&=\frac12J\Psi\left([J[JU,X],Y]+[X,J[JU,Y]]\right)\\
 &\quad-\frac a2J[[JU,X],JY]-\frac a2J[JX,[JU,Y]]+a^2[JX,JY]-ab[JX,Y]-ab[X,JY]\\
 &=\frac12J\Psi\left([J[JU,X],Y]+[X,J[JU,Y]]\right)-\frac a2J[JU,[X,JY]]-\frac a2J[JU,[JX,Y]]\\
 &\quad +a^2J[JX,Y]+a^2J[X,JY]-ab[JX,Y]-ab[X,JY]\\
 &=\Psi\left(\frac12J[J[JU,X],Y]+\frac12J[X,J[JU,Y]]-aJ[X,JY]-aJ[JX,Y]\right)=0\,,
 \end{split}
\]
since the argument of $\Psi$ lies in $\f g_J'$.
Finally, we compute
 \[
 \begin{split}
 [\Psi(X),\Psi(JY)]&=\frac14[[JU,X],J[JU,Y]]-\frac a2[[JU,X],JY]-\frac a2[X,J[JU,Y]]+a^2[X,JY]\\
 &\quad  -\frac b2[JX,J[JU,Y]]+ab[JX,JY]-b^2[JX,Y]\\
 &=\frac 14[JU,[X,J[JU,Y]]]-\frac a2[JU,[X,JY]]-\frac a2[X,J[JU,Y]]+a^2[X,JY]\\
 &\quad -\frac b2J[X,J[JU,Y]]-\frac b2J[JX,[JU,Y]]+abJ[X,JY]+abJ[JX,Y]-b^2[JX,Y]\\
 &=\frac12\Psi([X,J[JU,Y]])-a\Psi([X,JY])-bJ\left(\frac12[JU,[JX,Y]]-a[JX,Y]-bJ[JX,Y]\right)\\
&=\Psi\left(\frac12[X,J[JU,Y]]-a[X,JY]-bJ[JX,Y]\right)\,.
 \end{split}
 \]
It follows that $[\Im \Psi, \Im \Psi]\subseteq \Im \Psi$, and hence $\Im \Psi$ is a subalgebra, as desired.
\end{proof}

\begin{rmk}
The appearance of the cohomological obstruction in Theorem \ref{levihol} is somewhat unsurprising. Indeed, a possible proof of the classical Levi--Malcev decomposition uses the fact that the second cohomology group with values in any representation of a semisimple Lie algebra always vanishes, see \cite[Theorems 3.12.1, 3.14.1]{Var84}. The difference here is that, in our framework, the vanishing of the relevant cohomology group is not always guaranteed and must therefore be imposed as an additional assumption. For instance, it is easy to check that Example \ref{noLevi} violates the condition $H^2(\g,JU)=0$. 
\end{rmk}

In order to extend Theorem \ref{levihol} to the general case we need some preparations.  Let $\f g $ be a $2$-step solvable Lie algebra, we know that $\f g_J'$  is a $J$-invariant ideal in $\f g'+J\f g'$.  
We can then make use of Theorem \ref{Leviabelian} to infer that there exists a $J$-invariant subalgebra $\h$ satisfying 
$$
\frac{\f g'+J\f g'}{\f g_J'}=\f h \ltimes \frac{{\rm Rad}(\f g'+J\f g')}{\f g_J'}\,.
$$
If $\f g $ is not $J$-solvable, then, as in Corollary \ref{lemspettroB}, we observe that
\begin{equation}\label{adu^2=2adu}
\ad_{JU}^2=2\ad_{JU}\,.
\end{equation}
In particular, $\ad_{JU}$ is diagonalisable over $\frac{{\rm Rad}(\f g'+J\f g')}{\f g_J'}$ and
\begin{equation}\label{specU}
\mathrm{Spec}\left(\ad_{JU}|_{\frac{{\rm Rad}(\f g'+J\f g')}{\f g_J'}} \right) \subseteq \{0,2\}\,.
\end{equation}
Let $V_0:=\ker\left(\ad_{JU}|_{\frac{{\rm Rad}(\f g'+J\f g')}{\f g_J'}}\right)$ and,  denoted by $\pi\colon \f g'+J\f g'\to \frac{\f g'+J\f g'}{\f g_J'} $ the canonical projection onto the quotient,  we define
\begin{equation}\label{defkon}
\f q:=\pi^{-1}(V_0)\cap \f g'=\{ X\in \mathrm{Rad}(\g'+J\g')\cap \g' \mid [JU, X] \in \g'_J\}\,.
\end{equation}
The next is a technical lemma we will use later.

\begin{lem}\label{trigo}
Let $\f g$ be a $2$-step solvable Lie algebra such that $H^2(\g, JU)=0$.  Suppose $S,T\in \g'$ satisfy
$$
[JU, S]= aS+bT + Z\,,\qquad
[JU, T]= -bS+aT +Z'\,,
$$
for some $a,b\in \R$ with $(a, b)\ne (2, 0)$ and $Z,Z'\in \g'_J$. If $X\in\f g$ is such that $[X,Z],[X,Z']\in \mathfrak{q}$, then  $[X,S],[X,T]\in  \mathfrak{q} $.
\end{lem}
\begin{proof}
We start by observing that the hypothesis $H^2(\g, JU)=0$ and Corollary \ref{corobgprimo} put us in the position to apply Theorem \ref{levihol} and infer that 
$$
\g'+J\f g'=\f h\ltimes {\rm Rad}(\g'+ J \f g')\,, \qquad \f h\simeq \frac{\f g'+J\f g'}{{\rm Rad}(\g'+ J\f g')}.
$$
We first prove that if $[X,Z],[X,Z']\in {\rm Rad}(\g'+J\g')$, then $[X,S],[X,T] \in {\rm Rad}(\g'+J\g')$. Surely $[X,S],[X,T]\in \g'+J\g' $ and we can decompose
$$
[X,S]=\sum_{i=1}^{\dim \f h'} \alpha_i Y_i + W_1\,, \qquad 
[X,T]=\sum_{i=1}^{\dim \f h'} \beta_i Y_i+ W_2 \,,
$$
where  $\{Y_1, \ldots, Y_{\dim \f h'}\}$ is a basis of $\f h'$ and $ W_1, W_2\in {\rm Rad}(\g'+J\g')$. 
Then, since $S\in\f g'$ and ${\rm Rad}(\f g'+J\f g')$ is an ideal in $\f g'+J\f g'$, $2$-step solvability and \eqref{ad_JU2Id} imply
\[
\begin{aligned}
0=&\, [[X,JU],S]=[[X,S],JU]+[X,[JU,S]]
=\sum_{i=1}^{\dim \f h'}\left(-2\alpha_i+a\alpha_i+b\beta_i\right)Y_i+W_3\,,\\
0=&\, [[X,JU],T]=[[X,T],JU]+[X,[JU,T]]=\sum_{i=1}^{\dim \f h'} \left(-2\beta_i-b\alpha_i+a\beta_i\right)Y_i+W_4\,,
\end{aligned}
\] 
where $W_3, W_4\in {\rm Rad}(\g'+J\g')$.
Hence, for all $i=1, \ldots, \dim \f h'$,  we obtain the system of equations
\begin{equation}\label{sefattolastoria}
\begin{cases}
(a-2)\alpha_i+b\beta_i=0\,, \\
(a-2)\beta_i-b\alpha_i=0\,.
\end{cases}
\end{equation}
Assume by contradiction that $[X,S]\notin \mathrm{Rad}(\g'+J\g')$, then there exists $i=1, \ldots, \dim \f h'$ such that  $\alpha_i\ne 0$. Therefore, from the second equation in \eqref{sefattolastoria} we have $b=(a-2)\frac{\beta_i}{\alpha_i}$, which, plugged in the first equation, 
forces $a=2$ and consequently $b=0$ which is against our assumptions. We conclude similarly if $[X,T]\notin \mathrm{Rad}(\g'+J\g')$.

Now we prove the lemma. Suppose $[X,Z],[X,Z'] \in \mathfrak{q}$ then,  by the first part of the proof,  we know that $[X,S],[X, T]\in {\rm Rad}(\g'+J\g')$ and we may write
$$
[X,S]=P+Z_1\,,\qquad
[X,T]=Q+Z_2\,,    
$$
with $P,Q\in V$, for some complement $V$ of $\f g_J'$ in ${\rm Rad}(\f g'+J\f g')$, and $Z_1, Z_2\in\f g_J'$. By repeating  the computation done above, we get
\[
\begin{cases}
[JU,P]= aP+bQ + [X,Z]+Z_3\,,\\
[JU,Q]= -bP+aQ +[X,Z']+Z_4\,,\\
\end{cases} 
\]
for some $Z_3, Z_4\in\f g_J'$. Applying $\ad_{JU}$ to these equations and using the assumption that $[X,Z],[X,Z']\in \mathfrak{q}$,  we get
\begin{equation}\label{pezzodistoria}
\begin{cases}
[JU, [JU,P]]= a[JU,P]+b[JU,Q] +Z_5\,,\\
[JU, [JU,Q]]= -b[JU,P]+a[JU,Q] +Z_6\,,\\
\end{cases} 
\end{equation}
with $Z_5,Z_6\in \g'_J$. Now, either $[JU,P],[JU,Q] \in \g'_J $, meaning $[X,S],[X,T]\in \mathfrak{q}$, or $\lambda=a+\sqrt{-1}b$ is an eigenvalue of $\ad_{JU} \vert_{\frac{{\rm Rad}(\g'+J\g')}{\g'_J}}$ with eigenvector $[JU,P]-\sqrt{-1}[JU,Q]$.  Thus,  by \eqref{specU} we must have $\lambda\in \{0,2\}$. On the other hand, since  $(a,b)\neq (2,0)$,  we necessarily have $a=b=0$.  Using \eqref{adu^2=2adu},  we conclude from \eqref{pezzodistoria} that $[JU,P],[JU,Q] \in \g'_J $, as desired.
\end{proof}
 With this settled, we are now ready to prove the following.
\begin{prop}\label{lastoria}
Let $\f g$ be a $2$-step solvable Lie algebra such that $H^2(\g, JU)=0$. Then
\[
\g'_J \subseteq {\rm Rad}(\g)\,.
\]
\end{prop}
\begin{proof}
First, we aim to prove that $\mathfrak{q}$ is an ideal, where $\mathfrak{q}$ is defined in \eqref{defkon}. Notice that, by 2-step solvability we have
\[
\ad_{JU}\vert_{(\g'_J)^{1,0}}=\ad_{JU}\vert_{(\g'_J)^{1,0}}+\sqrt{-1}\ad_{U}\vert_{(\g'_J)^{1,0}}=\ad_{JU+\sqrt{-1}U}\vert_{(\g'_J)^{1,0}}
\]
and so, by Lie's theorem,  there exists a basis $\{W_1,\dots,W_s\}$ of $(\g'_J)^{1,0}$ with respect to which $\ad_{JU}\vert_{(\g'_J)^{1,0}}$ is lower triangular. In particular, the eigenvalues $\lambda_i\in \C$ of $\ad_{JU}\vert_{(\g'_J)^{1,0}}$ satisfy
\[
[JU, W_k]=\lambda_k W_k+ V_k\,, \qquad V_k\in \bigoplus_{i<k} \langle W_i\rangle \qquad k=1,\dots,s\,.
\]
Writing $a_k:=\Re(\lambda_k)$, $b_k:=\Im(\lambda_k)$, and $Z_k:=\Re(W_k)$,  this is equivalent to
\[
\begin{cases}
[JU, Z_k]= a_kZ_k+b_kJZ_k+ \tilde Z_k\,, & \tilde Z_k\in \bigoplus_{i<k} \langle Z_i,JZ_i \rangle\,, \\
[JU,JZ_k]= -b_kZ_k+a_kJZ_k+\tilde Z'_k\,, & \tilde Z'_k\in \bigoplus_{i<k} \langle Z_i,JZ_i\rangle \,,\\
\end{cases} \qquad k=1,\dots,s\,.
\]
Note that $\g'_J \subseteq \mathfrak{q}$. We first prove that $[\g,\g'_J]\subseteq \f q$. In order to do this, we use induction on $k$, and show that $[\g,Z_k],[\g,JZ_k]\subseteq \mathrm{Rad}(\g'+J\g')$. We observe that the condition $H^2(\g,JU)=0$ entails that $(a_k,b_k)\neq (2,0) $, for all $k=1,\dots,s$. 
For $k=1$, using Lemma \ref{trigo} with $S=Z_1, T=JZ_1, Z=Z'=0$,  we see that we have $[\g,Z_1],[\g,JZ_1]\subseteq \f q$. Now, suppose that $[\g,Z_{i}], [\g,JZ_{i}]\subseteq  \f q$,  for all $i=1,\dots,k-1$, and let us prove the claim for $k$. By the inductive assumption,  we can again apply Lemma \ref{trigo} with $S=Z_k$, $T=JZ_k$, $Z=\tilde Z_k$, and $Z'=\tilde Z'_k$. Hence, we conclude that $[\g,Z_k],[\g,JZ_k]\in \f q$ for all $k$, meaning that $[\g,\g'_J]\subseteq \f q$. Knowing this, it now follows immediately from Lemma \ref{trigo} that $[\g,\mathfrak{q}]\subseteq \mathfrak{q}$.

Let $\f k$ be the ideal of $\g$ generated by $\g'_J$, i.e.
\[
\f k=\g'_J+[\g,\g'_J]+[\g,[\g,\g'_J]]+\dots
\]
Since $\mathfrak{q}$ is an ideal, we have $\mathfrak{k} \subseteq \mathfrak{q}\subseteq \mathrm{Rad}(\g'+J\g')\cap \g' $. Note that, for every $X\in \g$, $Y\in \mathfrak{k}$ we have $[X,JY]+[JX,Y]\in \g '_J $ and thus $\f k+ J \f k$ is an ideal of $\g$ contained in $\mathrm{Rad}(\g'+J\g')$. In particular, $\f k + J\f k $ is a $J$-solvable ideal of $\f g$ which contains $\f g_J'$, implying that $\f g_J'\subseteq {\rm Rad}(\f g)$.
\end{proof}

The above proposition allows us to draw the following corollary, which in particular shows that if $H^2(\g,JU)=0$, then any semisimple complex structure is abelian.

\begin{cor}\label{Ilcorollario}
Let $\f g$ be a $2$-step solvable Lie algebra such that $H^2(\g, JU)=0$.
Then
\[
{\rm Rad}(\mathcal O(\f g_J'))={\rm Rad}(\f g)\cap \mathcal O(\f g_J')\,, \qquad \g=\mathcal O(\f g_J')+ {\rm Rad}(\f g)\,.
\]
Moreover, if $J$ is semisimple, then it is abelian.
\end{cor}
\begin{proof}
Combining Proposition \ref{lastoria} and Lemma \ref{Lemmino} immediately gives the first part. In turn, if $J$ is semisimple, this yields $\g'_J\subseteq {\rm Rad}(\mathcal O(\g_J'))={\rm Rad}(\g)\cap \mathcal O(\g_J')=0$.
\end{proof}

We can now  prove Theorem \ref{Main1}.

\begin{thm}\label{storiafatta}
Let $\f g$ be a $2$-step solvable Lie algebra and suppose that $H^2(\g, JU)=0$. Then there exists a unique $J$-semisimple subalgebra $\f h$  of $\f g$ such that
$$
\f g=\f h\ltimes {\rm Rad}(\f g)\,. 
$$
\end{thm}
\begin{proof}
Thanks to Corollary \ref{Ilcorollario}, we know that ${\rm Rad}(\mathcal{O}(\g'_J)) \subseteq {\rm Rad}(\g)$. Suppose ${\rm Rad}(\mathcal{O}(\g'_J))={\rm Rad}(\g)$, which implies ${\rm Rad}(\g)\subseteq \mathcal{O}(\g'_J)$. Similarly to Step 2 of Theorem \ref{Leviabelian} we deduce that $\g=\mathcal{O}(\g'_J)$. We are now in the position to apply Theorem \ref{levihol} and conclude this case.

Therefore we only need to deal with the case ${\rm Rad}(\mathcal{O}(\g'_J))\subset {\rm Rad}(\g)$. 
We know that $\g/{\rm Rad}(\mathcal{O}(\g'_J))$ has an abelian  complex structure. By Theorem \ref{Leviabelian}, there exists a $J$-semisimple subalgebra $\f s$ such that
\[
\frac{\f g}{{\rm Rad}(\mathcal{O}(\g'_J))}=\mathfrak{s}\ltimes \frac{{\rm Rad}(\g)}{{\rm Rad}(\mathcal{O}(\g'_J))}\,, \quad \f s\simeq \frac{\f g}{{\rm Rad }(\f g)}\,.
\]
Consider the $J$-invariant subalgebra $\hat{\f s}=\pi^{-1}(\f s)$ of $\g$, where $\pi\colon \f g \to \frac{\f g}{{\rm Rad}(\mathcal{O}(\g'_J))}$ is the canonical projection onto the quotient. As in Step 1 of Theorem \ref{Leviabelian},  we deduce ${\rm Rad}(\hat{\f s})={\rm Rad}(\mathcal{O}(\g'_J))$. This, together with the fact that $\f s=\hat{\f s}/{\rm Rad}(\mathcal{O}(\g'_J))$ is $J$-perfect, implies that $\hat{\f s}\subseteq \mathcal{O}(\g'_J)$. 
Using this last fact and the fact that $\hat{\f s}$ is a subalgebra, we see that $\hat{ \f s}$ acts holomorphically on $\hat{\f s}_J'$. 
Now, by Corollary \ref{Ilcorollario}, we observe that 
$$
\begin{aligned}
    \frac{\hat{\f s}}{{\rm Rad}(\hat{\f s})}\simeq \frac{\f g}{{\rm Rad}(\f g)}=\frac{\mathcal{O}(\g'_J)+{\rm Rad}(\f g)}{{\rm Rad}(\f g)}\simeq \frac{\mathcal{O}(\g'_J)}{{\rm Rad}(\f g)\cap \mathcal{O}(\g'_J)}=\frac{\mathcal{O}(\g'_J)}{{\rm Rad}(\mathcal{O}(\g'_J))}
\end{aligned}
$$
Thus, we are in the position to apply Lemma \ref{unnosomicasai} to conclude that $H^2(\hat{\f s}, JU_{\hat{\f s}})=0$.
Hence, $\hat{\f s}$ satisfies the hypotheses of Theorem \ref{levihol}, which allows us to write $\hat{\f s}=\f h\ltimes {\rm Rad}(\mathcal{O}(\g'_J)).$ It follows  that $\f g=\f h\ltimes {\rm Rad}(\f g)$, which concludes the existence part of the theorem.

The only thing left to prove is uniqueness of the $J$-adapted Levi subalgebra. The argument is by induction on $\dim_{\C}\f g$ and goes as in the proof of Theorem \ref{Leviabelian}, replacing $\f z(\g)$ with $\g'_J$. The only case that is slightly different is when $\g$ is $J$-perfect and $0\neq \g'_J={\rm Rad}(\g)$, so we provide some details. Let $\h_1,\h_2$ be $J$-adapted Levi subalgebras of $\g$.
Let $p_{\f h_1}$ and $p_{\g'_J}$ be the holomorphic projections of $\g$ onto $\f h_1$ and $\g'_J$, respectively. For any $Z_1, Z_2\in \f h_2$,  we have
$$
p_{\f h_1}([Z_1, Z_2])+ p_{\f g'_J}([Z_1, Z_2])=[Z_1, Z_2]=[p_{\f h_1}(Z_1), p_{\f h_1}(Z_2)]+[p_{\f h_1}(Z_1), p_{\f g'_J}(Z_2)]+[p_{\f g'_J}(Z_1), p_{\f h_1}(Z_2)]\,,
$$
and thus $p_{\f h_1}$ is a holomorphic Lie algebra homomorphism and
\begin{equation}\label{boloseipazzo}
p_{\f g'_J}([Z_1, Z_2])=[p_{\f h_1}(Z_1), p_{\f g'_J}(Z_2)]+[p_{\f g'_J}(Z_1), p_{\f h_1}(Z_2)]\,.
\end{equation}
Now, choosing $Z_1=JU_{\h_2}$ and $Z_2\in \h_2'$, we get from \eqref{boloseipazzo}
\[
2p_{\f g'_J}(Z_2)=[JU_{\h_1}, p_{\f g'_J}(Z_2)]\,,
\]
where we used $2$-step solvability and the fact that $p_{\f h_1}(JU_{\h_2})=JU_{\f h_1}$, thanks to Lemma \ref{ILGRANDEJU}, since $p_{\f h_1}\vert_{\f h_2}$ is a holomorphic Lie algebra isomorphism. But then $p_{\f g'_J}(Z_2)$ must vanish, since otherwise $2$ would be an eigenvalue for $\ad_{JU_{\f h_1}}\vert_{\g'_J}$, which is prohibited by our assumptions. As $\f h_2$ is $J$-perfect we infer $p_{\f g'_J}|_{\f h_2}=0$ and thus $\f h_2=\f h_1$, as desired. 
\end{proof}

We now show that whenever there is an SKT metric compatible with the complex structure the obstruction to the existence of the Levi--Malcev decomposition vanishes.

\begin{thm}\label{SKTLevi}
Let $\f g$ be a $2$-step solvable SKT Lie algebra.  Then $H^2(\g,JU)=0$. As a consequence,
$$
\f g=\f{aff}(\R)^p\ltimes {\rm Rad}(\f g)\,,
$$
for some $p\ge 0$.
\end{thm}
\begin{proof}
We know that $\mathcal O(\f g_J')$ is a $J$-invariant ideal of $\f g$ which inherits an SKT metric. 
Since $\f g_J'$ is a $J$-invariant ideal of $\mathcal{O}(\g'_J)$, we can use Theorem \ref{Leviabelian} to infer that 
$$
\frac{\mathcal{O}(\g'_J)}{\f g_J'}=\f h\ltimes \frac{{\rm Rad}(\mathcal{O}(\g'_J))}{\f g_J'}\,, \quad \f h\simeq \frac{\mathcal O (\f g_J')}{{\rm Rad}(\mathcal O(\f g_J'))}\,.
$$
We consider the subalgebra $\f  h_U=\langle U, JU \rangle$  of $\f h$ and its preimage $\pi^{-1}(\f h_U)$ via the canonical projection onto the quotient $\pi\colon \mathcal{O}(\g'_J) \to \mathcal{O}(\g'_J)/ \f g_J'$. Observe that $\pi^{-1}(\f h_U)$ is an almost-abelian $J$-invariant subalgebra of $\f g$. Indeed, $\f a:=\pi^{-1}(\langle  U\rangle)=\pi^{-1}(\f h_U')$ is an abelian ideal of $\pi^{-1}(\f h_U)$, since $U\in  (\mathcal{O}(\g'_J)/\f g_J')'=\mathcal{O}(\g'_J)'/\f g_J'$ and $\f g $ is $2$-step solvable. Since $\mathcal{O}(\g'_J)$ is SKT, $\pi^{-1}(\f h_U)$ admits an SKT metric $g$. We consider an orthogonal complement $\langle U', JU'\rangle$ of $\f g_J'$ in $\pi^{-1}(\f h_U)$ with respect  to   $g$. Clearly $U'=bU+X$, for some $X\in\f g _J'$ and $b\in \R$.  Therefore, we have $[JU',U' ]=2bU'+ v, $ for some $v\in\f g_J'$, and thus $\ad_{JU'}|_{\f g_J'}=b\ad_{JU}|_{\f g_J'}$.  We can further suppose, up to scaling, that $|U'|_g=1$.  Now, we are in the position to apply \cite[Lemma 4.8]{AL19}  or \cite[Main Theorem (iv)]{Asia} with $a=2b$, to infer that the real parts of the eigenvalues of $\ad_{JU}|_{\f g_J'}$ are either $0$ or $-1$. This concludes the proof of the first part. 

We can now apply Theorem \ref{storiafatta} to obtain 
$$
\f g=\f h\ltimes {\rm Rad}(\f g)\,,
$$
where $\f h $ is a $J$-semisimple subalgebra. Using Corollary \ref{corsymmetricdecom}, we have that $\f h$ is a direct sum of $p$ copies of $\f{aff}(\R)$ and $q$ copies of $\f{aff}(\C)$. On the other hand, $\f h$ is again SKT and hence $q=0$, since $\f{aff}(\C)$ does not admit any SKT metric.
\end{proof}

\begin{rmk}
Similar decompositions were studied in the literature in relation with the existence of special Hermitian metrics. For instance, in \cite{FKV15} the non-existence of SKT metrics was proved in the case of certain semidirect products. However, usually not all the assumptions in \cite[Theorem 1.1]{FKV15} are satisfied in our framework.
\end{rmk}
As a first application of Theorem \ref{SKTLevi}, we can characterize unimodular, 2-step solvable SKT Lie algebras with ${\rm Rad}(\f g)=\f g'_J$ as \emph{OT-like}, see \cite[Definition 3.18]{ShuhoOT}. Examples of said Lie algebras arise as Lie algebras of the universal covers of Oeljeklaus-Toma manifolds, see \cite{Ka12}.  The existence of SKT metrics on OT manifolds was characterized in \cite{ADOS24,O22}. In view of the latter, we will say that a Lie algebra is \emph{SKT OT-like} if it is OT-like and, for any $i=1, \ldots, s$,  there exists a unique $j_i=1, \ldots, t$ such that ${\rm Re}(c_{ij_i})=-\frac12$ and, $j_i\ne j_{i'}$ whenever $i\ne i'$. We first consider the following easy example.
\begin{es}
Consider the Lie algebra $\f g=\langle e_1, \ldots, e_6\rangle$ defined by the structural equations:
$$
\begin{aligned}\label{esnoSKT}
[e_2, e_1]=&\, e_1, \quad[e_4, e_3]=e_3, \quad[e_2, e_5]=-\frac12e_5+ae_6, \quad[e_2, e_6]=-\frac12e_6-ae_5,\\ [e_4, e_5]=&\, -\frac12e_5+be_6, \quad[e_4, e_6]=-\frac12e_6-be_5\,,
\end{aligned}
$$  with $a, b\in \R$ and  complex structure given by $Je_1=e_2,$ $Je_3=e_4$ and $Je_5=e_6$. We observe that 
$$
dd^c(e^{56})=e^{1456}-e^{2356}\,, \quad \iota_{e_5}\iota_{e_6}dd^c e^{ij}=0, \, \quad (i, j)\ne (5, 6)\,.
$$  Hence, $\f g$ cannot admit any SKT metric.
\end{es}
\begin{prop}
Let $\f g$ be a unimodular,  SKT,   $2$-step solvable Lie algebra such that ${\rm Rad}(\f g)=\f g'_J$. Then, $\f g$ is SKT OT-like.
\end{prop}
\begin{proof}
 Using Theorem \ref{SKTLevi}, we know that 
 $$
\f g=\f{aff}(\R)^s\ltimes \f g_J'\,.
 $$ We  consider  $\{X_1, JX_1, \ldots, X_s, JX_s\}$ the standard basis of $\f{aff}(\R)^s$, see \eqref{baseincredibile},  and denote with $t=\dim _{\C}\f g_J'$. Now, for any $i=1, \ldots, s$,  we have that $\langle X_i, JX_i \rangle\ltimes \f g_J'$ is an almost-abelian ideal of $\f g$, which hence inherits an SKT metric. Using \cite[Lemma 4.8]{AL19} or \cite[Main Theorem (iv)]{Asia}, we can infer that $\ad_{JX_i}|_{\f g_J'}$ is normal, thus diagonalizable over $\C$, $  [\ad_{JX_i}|_{\f g_J'},  J]=0$ and ${\rm Re}({\rm Spec}(\ad_{JX_i}|_{\f g_J'}))\subseteq\{0,-\frac12\}$. Moreover, using $2$-step solvability, we have that $\{\ad_{JX_1}|_{\f g_J'},\ldots,  \ad_{JX_s}|_{\f g_J'}\}$ is a family of commuting and diagonalizable endomorphisms. Therefore, we can find a basis  $\{Y_1, JY_1, \ldots, Y_t, JY_t\}$ of $\f g_J'$ such that, for any $i=1, \ldots, s$ and $j=1, \ldots, t$, 
 $$
 \ad_{JX_i}Y_j=\varepsilon_{ij}Y_j+a_{ij}JY_j\,,\quad  \ad_{JX_i}JY_j=\varepsilon_{ij}JY_j-a_{ij}Y_j\,, 
 $$ where $a_{ij}\in \R$ and $\varepsilon_{ij}\in\{0,-\frac12\}$.  
 % The proof will be  completed if we prove that  for any $j=1, \ldots, t$ there exists a unique $i=1, \ldots, p$ such that $\varepsilon_{ij}=-\frac12$.
 Firstly, we show that,   for any $i=1, \ldots, s$, there exists a unique  $j_i=1, \ldots, t $ such that $\varepsilon_{ij_i}=-\frac12$.  Indeed, 
 % $I:=\{i\in \{1, \ldots, s\}\,\, |\, \, \varepsilon _{ij}=0\,, j=1, \ldots, t\}=\emptyset$. Indeed, if it was not the case, for any $i\in I$, ${\rm tr}(\ad_{JX_i})=1$,  against unimodularity of $\f g$. Hence, for any $i=1, \ldots, s$, there exists $j_i=1, \ldots, t $ such that $\varepsilon_{ij_i}=-\frac12$. We show that $j_i$ is unique. By contradiction, 
 let $k\ge 0$ be such that $\{j_{i_1}, \ldots, j_{i_k}\}\subseteq \{1, \ldots, t\}$ satisfy $\varepsilon_{ij_{i_r}}=-\frac12$, for every $r=1, \ldots, k$.  In this case, we have that 
 $$
 0={\rm tr}(\ad_{JX_i})=1-k\,,
$$forcing $k=1$.  Finally, we prove that, for any $i,i'=1, \ldots, s$ with $i\ne i'$ then, $j_i\ne j_{i'}$.  Suppose $j_i=j_{i'}$ and consider the $J$-invariant subalgebra $\f s=\langle X_i, JX_i, X_{i'}, JX_{i'}, Y_{j_i}, JY_{j_i}\rangle$. $\f s$  inherits an SKT metric but  it is isomorphic to the Lie algebra in Example \ref{esnoSKT} which is a contradiction.  The claim follows by considering $C=(c_{ij})$  where $c_{ij}=\varepsilon_{ij}+ \sqrt{-1}a_{ij}$ in \cite[Definition 3.18]{ShuhoOT}. 
% On the other hand, if we suppose that $p<t$, we have
%  $$
% \f k=\f{aff}(\R)^p\ltimes \langle  Y_{p+1}, JY_{p+1}, \ldots, Y_{t}, JY_{t}\rangle
%  $$ is a non-unimodular  $J$-invariant  ideal  of $\f g$. Thus, $p=t$, this conclude the proof.
 \end{proof}

We conclude this section by highlighting the fact that not all SKT OT-like Lie algebras are the Lie algebra of the universal cover of a Oeljeklaus-Toma manifold.  
\begin{es}
Consider the unimodular Lie algebra $\f g=\langle e_1, \ldots, e_6\rangle$ defined by the structural equations:
$$
\begin{aligned}\label{esnoSKT}
[e_2, e_1]=&\, e_1, \quad[e_2, e_3]=-\frac12e_3+ae_4, \quad[e_2, e_4]=-\frac12e_4-ae_3,\\ [e_2, e_5]=&\, be_6, \quad[e_2, e_6]=-be_5\,,
\end{aligned}
$$
with $a, b\in \R$ and complex structure given by $Je_1=e_2,$ $Je_3=e_4$ and $Je_5=e_6$. We observe that 
$$
dd^c(e^{12}+e^{34}+e^{56})=0\,.
$$
Hence, $\f g$ admits a SKT metric. On the other hand,  $\f g= \mathfrak{aff}(\R)\ltimes \g'_J$ with $\dim_{\C}\f g_J'=2$, giving that $\f g$ cannot be the Lie algebra of the universal cover of a SKT Oeljeklaus-Toma manifold by \cite[Corollary 3]{ADOS24}.
\end{es}

\section{Solvable complex structures}\label{secsolv}
In this section, we work towards the proof of Theorem \ref{mainfv} and Theorem \ref{mainsktcompletely}.
\subsection{\texorpdfstring{$J$}{J}-solvability and Chern--Ricci flatness}
In this subsection, we explore the relation between $J$-solvability of a Lie algebra and the vanishing of the form $\eta$, defined in \eqref{defn_eta}. 
For starters, we relate $\eta$ to the $J$-solvable radical in a class of Lie algebras that includes $J$-perfect ones. This will be useful in what follows.

\begin{prop}\label{kerrhorad}
Let $\f g= \f a+ J \f a  $ be a solvable Lie algebra such that $\f a $ is an abelian ideal and $\f g_J'\subseteq \f a_J $. Then,  $\f g'\subseteq \f a$,  $\f a_J$ is an ideal in $\f g $ and
$$
\eta=\pi^*\eta_{\f g/\f a_J}\,,
$$
where $\pi\colon \f g \to \f g/\f a_J $ is the canonical projection onto the quotient. Moreover, $\ker \eta={\rm Rad}(\f g)$.
\end{prop}
\begin{proof}
Firstly,   we prove that $\f g'\subseteq \f a $. Since $\f a$ is an ideal,  it is sufficient to prove that $[J\f a,J\f a]\subseteq \f a $. Fix $X, Y\in\f a$,
exploiting integrability, we have 
$$
\f g'\ni [JX, JY]=J([JX, Y]+[X, JY])\in J \f g'\,, 
$$
concluding that $[JX, JY]\in\f g_J'\subseteq \f a_J$. 
Similarly to the proof of Lemma \ref{Lem2stepideal} we deduce that $\f a_J$ is an ideal.
In addition, using that $\f a_J$ is abelian, for every $X\in \f a_J$ and $Y\in \g$ we have $\Im(J\ad_{[X, Y]})\subseteq \f a_J$ and $J\ad_{[X, Y]}|_{\f a_J}=0$. Consequently
$$
\eta(X, Y)=-\frac12{\rm tr}(J\ad_{[X,Y]})=0\,, \qquad X\in \f a_J\,,\,\,Y \in\f g\,,
$$
showing that $\f a_J \subseteq \ker \eta$. Consider a Hermitian metric $g$ and fix $Y\in (\f a_J)^{\perp}\cap \f a$, where the orthogonal is with respect to $g$. Since $\f a $ is an abelian ideal, $[Y, JY]\in\f a$ and we have that $\ad_{[Y, JY]}\vert_{\f a_J}=0$. Thus
\[
\begin{aligned}
\eta(Y, JY)&=-\frac12 \tr(J\ad_{[Y,JY]} \vert_{\f a_J})- \frac12\tr(J \ad_{[Y,JY]}\vert_{(\f a_J)^\perp})=-\frac12\tr(J\ad_{[Y,JY]^{\perp}}\vert_{ (\f a_J)^{\perp}})\,,
\end{aligned}
\]
where $[Y,JY]^{\perp}$ is the orthogonal projection on $(\f a_J)^{\perp}$.  Thus, we obtain that 
$$
\eta(Y, JY)
=-\frac12{\rm tr}^{\f g/\f a_J}(J\ad_{[\pi(Y), J\pi(Y)]})\,. 
$$
The claim that $\eta=\pi^*\eta_{\g/\f a_J}$ follows from the above equation and from the inclusion $\f a_J\subseteq \ker \eta$. 

Furthermore, we note that $\ker \eta=\pi^{-1}(\ker\eta_{\f g/\f a_J})=\pi^{-1}(\f n_J)$, where $\f n_J $ is the maximal nilpotent $J$-invariant ideal of $\f g/\f a_J$. Using that $\pi$ is a holomorphic Lie algebra homomorphism, we get that $\ker \eta $ is a $J$-invariant ideal of $\f g$. Finally, by Lemma \ref{lemmagrullo}  $\ker\eta $ is also $J$-solvable. Thus,  $\ker \eta\subseteq {\rm Rad}(\f g)$.

Next, we claim that $\f n_J={\rm Rad}(\f g/ \f a_J)$. Indeed, since $\f g_J'\subseteq \f a_J$, we have that $\f g/\f a_J$ has an abelian complex structure $J$. Moreover, clearly, $\f g/\f a_J=\f a/\f a_J\oplus J \f a/\f a_J.$  In addition, since $\f g'\subseteq \f a$, we obtain $(\f g/\f a_J)'_J=(\f g'/\f a_J)\cap J(\f g'/\f a_J)\subseteq (\f a/\f a_J)\cap J(\f a/\f a_J)=0$. Then, using Theorem \ref{propsemidir},  we conclude the proof of the claim. In particular, since 
$$
\frac{\f g}{\ker \eta}\simeq \frac{\f g/\f a_J}{\pi^{-1}(\f n_J)/\f a_J}\simeq \frac{\f g/\f a_J}{\f n_J}\,,
$$this gives that $\f g/ \ker \eta$ is $J$-semisimple and hence ${\rm Rad}(\f g)=\ker \eta$, which concludes the proof. 
\end{proof}

Proposition \ref{kerrhorad} allows us to relate $J$-solvability with the vanishing of $\eta$, as follows. 

\begin{prop}\label{rhocore}
Let $\f g$ be a  $2$-step solvable Lie algebra endowed with a complex structure $J$. Then $\eta=0$ if and only if $\f g $ is $J$-solvable. In particular,  $\frac{\f g'+J\f g' }{\f g_J'}$ is nilpotent.
\end{prop}
\begin{proof}
It is easy to see that  $\eta|_{\f \g'+J\g' }=\eta_{\f g'+J\g' }\,$. We observe that, since $\f g $ is a $2$-step solvable Lie algebra, $\f g'$ is an abelian ideal of $\f g'+J\g'$.  
We are in the position to apply Proposition \ref{kerrhorad} to $\g'+J\g'$, and obtain, assuming $\eta=0$, that $\g'+J\g'={\rm Rad}(\f g'+J\g') $, which clearly gives that $\f g$ is $J$-solvable.

Viceversa, let us suppose that $J$ is solvable. In particular, by Proposition \ref{kerrhorad}, $\g'+J\g'$ is $J$-solvable and equal to $\ker \eta_{\f \g'+J\g'}=\pi^{-1}(\f n_J)$, where $\pi\colon \f g'+J\f g' \to (\f g'+J\f g')/\f g_J'$ is the canonical projection onto the quotient and $\f n_J$ is the maximal $J$-invariant nilpotent ideal of $(\f g'+J\g')/\f g_J'$. Thus, $( \f g'+J\g')/\f g_J'=\f n_J$  is  nilpotent. We will use this piece of information to conclude that $\eta=0$.  We know that $\eta=0$ is equivalent to require that $\beta|_{\f g'}=0$, where $\beta$ is the potential $1$-form of $\eta$, see \eqref{eqn_beta}. Since $\f g_J'$ is an abelian ideal in $\f g'+J\g'$ by Lemma \ref{Lem2stepideal}, we observe that $\g'_J \subseteq \ker \beta$. Now, let $X\in\g'$, we have
$$
\beta(X)=\frac12{\rm tr}(J\ad_X|_{\f g'+J\g'})=\frac12{\rm tr}^{\f n_J}(J\ad_{\pi(X)})+\frac12{\rm tr}(J\ad_X|_{\f g'_J})=\pi^*\beta_{\f n_J}(X)\,,
$$
where  we used that $\ad_X|_{\f g_J'}=0$, because $X\in \f g'$.
Now, thanks to \cite[Proposition 2.1]{LRV15}, we know that $\beta_{\f n_J}=0$, since $\f n_J$ is nilpotent. This concludes the proof. 
\end{proof}

 We are now ready to prove Theorem \ref{mainfv}.
\begin{thm}
Let $\g$ be a $2$-step solvable unimodular Lie algebra endowed with a complex structure $J$. If there exist both an SKT metric and a balanced metric compatible with $J$, then $J$ is solvable. In particular, there also exists a K\"ahler metric compatible with $J$.
\end{thm}
\begin{proof}
Applying Theorem \ref{SKTLevi} we immediately know that $ \f g=\f{aff}(\R)^p\ltimes {\rm Rad}(\f g),$ for some $p\ge 0$. Suppose by contradiction that $p\neq 0$ and  consider the holomorphic quotient map $\pi \colon \g \to \g/{\rm Rad}(\g)\simeq \f{aff}(\R)^p$. We know that   $\f{aff}(\R)^p$ is endowed with an exact K\"ahler form $-\eta $. Therefore, $-\pi^*\eta$ gives rise to a non-zero, semi-positive $(1,1)$-form which is exact. We can now use Corollary \ref{Michelsohn} to conclude $p=0$ and thus $\g$ must be $J$-solvable. In particular, Proposition \ref{rhocore} tells us that $\rho=0$. In this scenario we can then infer that there exists a K\"ahler metric invoking \cite{FV26}. 
\end{proof}

From Proposition \ref{kerrhorad} we can characterise the signature of the form $\eta$ on any $J$-perfect $2$-step solvable Lie algebra. 

\begin{cor}\label{signeta}
Let $ \f g $ be a $J$-perfect $2$-step solvable Lie algebra and let $\f g/\ker \eta\simeq \f{aff}(\R)^p\oplus \f{aff}(\C)^q$. Then,  the signature of $\eta$ is  $(2p+2q,2q, 2k)$ where $k=\dim_{\C}\ker \eta.$  In particular, $\eta \le 0$ if and only if $\f g/\ker \eta $ is the direct sum of $p$ copies of $\f{ aff}(\R).$
\end{cor}
\begin{proof}
  By Proposition \ref{kerrhorad} and  Theorem \ref{propsemidir}, we know that $ \eta=\pi^*\eta_{\f g/\f g'_J}$ and $\f g/\f g'_J=\f h\ltimes \f n_J$, where $\f n_J=\ker \eta_{\f g/\f g'_J}$ and $\f h\simeq \f g/\ker \eta$. Since $\eta_{\f g/\f g'_J}$ makes the $J$-simple  ideals of $\f h$ orthogonal, the conclusion follows from Corollary \ref{lemspettroB}.
\end{proof}

Example \ref{exsimplerigid3step} is a $J$-perfect $3$-step solvable Lie algebra and it is straightforward to check that $\rho=\eta=0$. Hence, one implication of Proposition \ref{rhocore} is no longer true for higher solvability steps.
\\
By means of Proposition \ref{rhocore}, we can also prove the converse of Corollary \ref{abelianpostLevi}, part \emph{(1)}.
\begin{cor}\label{besughi}
Let $\f g$ be a Lie algebra endowed with a solvable abelian complex structure $J$. Then $\f g'+J\f g'$ is unimodular.
\end{cor}
\begin{proof}
Using Proposition \ref{rhocore}, we have that $\eta=0$, which is equivalent to $\beta|_{\f g'}=0$. The claim then follows from the abelianity of $J$ and the observation that,  for any $X\in\f g'$, 
$ \beta(X)=-\frac12{\rm tr}(\ad_{JX})=0.$ 
\end{proof}

\subsection{SKT completely solvable Lie algebras}
 The aim of this subsection is to present the proof of Theorem \ref{mainsktcompletely}. We first prove the following proposition.

\begin{prop}\label{cazziarmati}
Let $\f g $ be a unimodular, solvable Lie algebra endowed with a complex structure such that  there exists a $J$-invariant abelian ideal $\f a $ satisfying $\f g_J'\subseteq \f a \subseteq C(\f g')$. Assume that $\f g$ admits an SKT metric $g$, then $\rho \le 0 $.
If $\rho=0$, $\ad_X|_{\f a}\in \f{su}(\f a, g )$, for any $X\in\f g$.  In these last hypotheses, if $\f g$ is completely solvable, then $\f a\subseteq \f z(\f g).$ In particular, if $\f g$ is $2$-step, then $\f g'+J\f g'$ is $2$-step nilpotent. 
\end{prop}
\begin{proof}
Using the same argument as in the proof of Proposition \ref{kerrhorad}, we infer that $\f a\subseteq \ker \rho.$
Now, we decompose $\f g=\f a^{\perp}\oplus \f a$, where the orthogonal complement is taken with respect to $g$. As a consequence,  for the claim to be true it suffices to prove that $\rho(X, JX)\le 0 $, for all $X\in\f a^{\perp}$. 
For any such  $X$, since $\f a \subseteq C(\f g')$ and $\f a$ is an ideal, we have 
$$
\begin{aligned}
\rho(X, JX)=&\, -\frac12{\rm tr}(J\ad_{[X, JX]}|_{\f a^{\perp}})-\frac12{\rm tr}(J\ad_{[X, JX]}|_{\f a})
= -\frac12{\rm tr}(J\ad_{[X, JX]}|_{\f a^{\perp}})\\
=&\, -\frac12{\rm tr}(J\ad_{[X, JX]^{\perp}}|_{\f a^{\perp}})=-\frac12{\rm tr}^{\f g/\f a}(J\ad_{\pi([X, JX])})\,,
\end{aligned}
$$
where  $\pi \colon \f g\to \f g/\f a$ is the canonical projection onto the quotient and $[X, JX]^{\perp}$ is the orthogonal projection of $[X, JX]$ onto $\f a^{\perp}$. Exploiting the fact that  $\f g/\f a$ has an abelian complex structure, since $\f g_J'\subseteq \f a$, we infer 
$$
\begin{aligned}
\rho(X, JX)=&\, -\frac12{\rm tr}^{\f g/\f a}(J\ad_{\pi([X, JX])})=\frac12{\rm tr}^{\f g/\f a}(\ad_{\pi(J[X, JX])})
=\frac12{\rm tr}(\ad_{J[X, JX]}|_{\f a^{\perp}})=-\frac12{\rm tr}(\ad_{J[X, JX]}|_{\f a})\,,
\end{aligned}
$$
where we used once again that $\f a$ is an abelian ideal of $\f g$ and unimodularity. 
Now,  we apply  Lemma \ref{SKTmaledetto} choosing $Y\in \f a$, and obtain 
$$
\begin{aligned}
0=&\,g([J[X, JX], JY], JY)+g([J[X, JX], Y],Y)
 +g([J[X, Y], X], JY)-g([J[X, JY], X], Y)\\
&\, +g([J[JX, Y], JX], JY)
 -g([J[JX, JY], JX], Y)
 -|[X, Y]|^2-|[X, JY]|^2\\
&\, - |[JX, Y]|^2-|[JX, JY]|^2\,.
\end{aligned}
$$
Summing over an orthonormal basis of $\f a$ with respect to $g$, we have 
$$
\begin{aligned}
{\rm tr}(\ad_{J[X, JX]}|_{\f a})=- {\rm tr}((J\ad_X)^2|_{\f a})-{\rm tr}((J{\ad_{JX}})^2|_{\f a})+|\ad_X|_{\f a}|^2+|\ad_{JX}|_{\f a}|^2\,.
\end{aligned}
$$
On the other hand, by integrability of $J$, we obtain
$$
\begin{aligned}
0&=\mathrm{tr}(J\ad_{[X,JX]}|_{\f a})
=\mathrm{tr}(J\ad_X\ad_{JX}|_{\f a})-\mathrm{tr}(J\ad_{JX}\ad_X|_{\f a})\\
&=-\mathrm{tr}(J\ad_X\ad_{X}J|_{\f a})+\mathrm{tr}(J\ad_XJ\ad_{X}|_{\f a})-\mathrm{tr}(J\ad_XJ\ad_{JX}J|_{\f a})-\mathrm{tr}(J\ad_{JX}\ad_X|_{\f a})
\\
&=\mathrm{tr}\left((\ad_X)^2|_{\f a}\right)+\mathrm{tr}\left((J\ad_X)^2|_{\f a}\right)\,.
\end{aligned}
$$ Hence, using the above equation, we get
$$
\begin{aligned}
\rho(X, JX)=-\frac12{\rm tr}(\ad_{J[X, JX]}|_{\f a})=&\, -\frac12\left( {\rm tr}((\ad_X)^2|_{\f a})+{\rm tr}(({\ad_{JX}})^2|_{\f a})+|\ad_X|_{\f a}|^2+|\ad_{JX}|_{\f a}|^2\right)\\
=&\, -2(|{\rm Sym}(\ad_X|_{\f a})|^2+|{\rm Sym}(\ad_{JX}|_{\f a})|^2)\le 0\,.
\end{aligned}
$$
If  $\rho=0$, then clearly 
$\ad_X|_{\f a}\in \f{so}(\f a, g)$.
The only thing remained to check is that $[\ad_X, J]|_{\f a}=0$, for any $X\in\f a^{\perp}$. Recalling the proof of Lemma \ref{riccihol}, we have 
$$
\frac12{\rm tr}([\ad_X|_{\f a}, J|_{\f a}]^2 )=\frac12{\rm tr}([\ad_X, J]^2 |_{\f a})={\rm tr}(J\ad_{[X, JX]}|_{\f a})=0\,.
$$
On the other hand, since $\ad_{X}|_{\f a}, J\in\f{so}(\f a, g)$, then $[\ad_X, J]|_{\f a}\in \f{so}(\f a, g)$. Hence, ${\rm tr}([\ad_X, J]^2|_{\f a})=0$ if and only if $[\ad_X, J]|_{\f a}=0$. Finally, if we assume that $\f g $ is completely solvable, then $\ad_X|_{\f a}=0$, for any $X\in\f g$, forcing $\f g_J'\subseteq\f a\subseteq \f z(\f g).$  Now, if $\f g$ is $2$-step with $\rho=0$, then $J$ is solvable and   $(\f g'+J\f g')/\f g_J'$ is nilpotent, by Proposition \ref{rhocore}. This, together with the fact that $\f g_J'$ is central, gives that $\f g'+J \f g'$ is nilpotent and hence $2$-step since it is SKT, see \cite{AN22}.
\end{proof}
\begin{rmk}\label{normaliziamodai}
Let $\f g$ be a $2$-step solvable Lie algebra.   We consider
$
N_{\f g}:=N_{\f g}(\f g_J')
$, the normaliser of $\f g_J'$ in $\f g$,  and note that $N_{\f g }$ is an ideal of $\f g$, since  $\f g'+J \f g' \subseteq N_{\f g}$ by Lemma \ref{Lem2stepideal}. Moreover, $N_{\f g}$ is $J$-invariant. Indeed, using \eqref{behaviourofJ2step},  for any $X\in N_{\f g}$ and $Y\in \f g_J'$, we have that 
$
[JX, Y]
\in \f g_J'
$ implying that $JX\in N_{\f g}$. Thus, if $\f g$ is a unimodular, SKT,  2-step solvable Lie algebra, then, $N_{\f g}$ has the same properties  and   contains $\f g_J'$ as an ideal. Most importantly, we note that $N_{\f g}$ satisfies the hypotheses of Proposition \ref{cazziarmati} with $\f a=\f g_J'$. 
\end{rmk}

 The next lemma collects some interesting facts on  solvable abelian complex structures.
\begin{lem}\label{b2,0+0,2}
Let $\f g$ be a  Lie algebra endowed with a solvable and abelian complex structure $J$. Then, the Killing form $B$ is of type $(2,0)+(0,2)$ and  $\ker B$ is a $J$-invariant ideal. Moreover,  the nilradical $\f n$ of $\f g$ is $J$-invariant. In particular, $\f g'+J\f g'$ is nilpotent. 
\end{lem}
\begin{proof}
 By Corollary \ref{besughi}, we know that $\f g'+J\f g'$ is unimodular.  Using this fact and Proposition \ref{propKilling}, part \emph{(2)}, we deduce that $B^{1,1}=0$, thereby proving the first statement. 
It follows that $\ker B$ is  $J$-invariant ideal, since $\f n \subseteq \ker B$. 
Let us now fix $X\in \f n$ and we prove that ${\rm tr}(\ad_{JX}^k)=0$, for any $k\ge 2$.   

For $k=2$, we have that $$0=B(X, X)=-B(JX, JX)=-{\rm tr}(\ad_{JX}^2)\,.$$
Now, let us suppose that $k\ge 3 $. The strategy is the same as in the proof of Proposition \ref{propKilling}, part \emph{(4)}.  Using \eqref{penis} and Lie's Theorem,  we can conclude, iterating $k-1$ times the argument, that  
$$
{\rm tr}(\ad_{JX}^k)={\rm tr}(\ad_{JX}^{k-2}\ad_{J[JX, X]})={\rm tr}(\ad_{J[JX, [JX, \ldots, [JX, X]]\ldots]})=0\,,
$$ 
by unimodularity of $\f g'+J\f g'$. Hence, ${\rm tr}(\ad_{JX}^k)=0$, for any $k\ge 2$, which implies $JX\in \f n$, and thus forces $\f n$ to be $J$-invariant.
\end{proof}

As consequence of the previous lemma we deduce the following.
\begin{cor}\label{abeliancompletely}
Let $(\f g, J)$ be a  Lie algebra endowed with an abelian complex structure $J$.
\begin{enumerate}
    \item If $\g$ is completely solvable and $J$ is solvable, then $\g$ is nilpotent.
    \item If $\g$ is of rigid type, then it is nilpotent.
\end{enumerate} 
\end{cor}
\begin{proof}
We prove \emph{(1)}. The proof of \emph{(2)} is similar once one observes that Lie algebras of rigid type are always unimodular, hence $J$-solvable by Corollary \ref{abelianpostLevi}, part \emph{(1)}. Applying Lemma \ref{b2,0+0,2}, we know that the Killing form is of type $(2,0)+(0,2)$. On the other hand, by complete solvability, we have 
$$
0\le {\rm tr}(\ad_X^2)=B(X, X)=-B(JX, JX)=-{\rm tr}(\ad_{JX}^2)\le 0\,,
$$
and hence ${\rm tr}(\ad^2_X)=0$, for any $X\in\f g$. Using again complete solvability, ${\rm tr}(\ad_X^2)=0$ implies that $\ad_X$ is nilpotent, for any $X\in\f g$, concluding the proof.
\end{proof}

We are now ready to prove Theorem \ref{mainsktcompletely}.
\begin{thm}\label{completelysolvSKT1}
Let $\f g$ be a unimodular, SKT, $2$-step completely solvable Lie algebra endowed with a solvable complex structure. Then $\f g$ is nilpotent. Consequently, any unimodular, SKT, $2$-step  completely solvable Lie algebra $\f g$ satisfies
$$
\f g=\f{aff}(\R)^p\ltimes \f n_J\,
$$ for some $p\ge 0$.
\end{thm}
\begin{proof}
As  for the first claim, by Remark \ref{normaliziamodai}, the normaliser $N_{\f g}$ of $\f g_J'$ in $\f g$ satisfies the hypotheses of Proposition \ref{cazziarmati}. Furthermore $N_{\f g}$ is $J$-solvable and completely solvable. 
Applying the last part of Proposition \ref{cazziarmati}, we have that $\f g_J'\subseteq \f z(N_{\f g})\cap (\f g'+J \f g')\subseteq \f z(\f g'+J \f g')$. Since $J$ is solvable, we know that $(\f g'+J\f g')/ \f g_J'$ is nilpotent and $\f g_J'\subseteq \f z(\f g'+J \f g')$, implying that $\f g'+J \f g' $ is nilpotent and SKT. Using \cite[Proposition 3.5]{EFV12},  we obtain that  $\f z(\f g'+J \f g')$ is $J$-invariant.  Hence, 
we are in the position to apply Proposition \ref{cazziarmati}  with $\f a=\f z(\f g'+J \f g' )$, which allows us to infer that $\f g_J'\subseteq\f z(\f g'+J\f g')\subseteq \f z(\f g)$. 
Therefore, taking the quotient by $\f g_J'$, we obtain a  completely solvable Lie algebra with a solvable abelian complex structure. Now, we can use Corollary \ref{abeliancompletely}, part \emph{(1)}, to conclude that $\f g/\f g_J'$ is nilpotent, which establishes the claim, using that $\f g_J'\subseteq \f z(\f g).$

The second claim is now a straightforward consequence of Theorem \ref{SKTLevi}, together with the first part of the proof applied to ${\rm Rad}(\f g)$. 
\end{proof}
 Non-existence  results for Hermitian symplectic structures  on completely solvable Lie algebras were proven in \cite{EFV12, FK16}.

We conclude this subsection inspecting the SKT condition on Lie algebras of real type.
\begin{lem}
Le $\f g$ be an SKT unimodular, $2$-step solvable Lie algebra of real type endowed with a solvable complex structure $J$. Then, $\f g'+J\f g' $ is $2$-step nilpotent.   
\end{lem}
\begin{proof}
We will again make use of Remark \ref{normaliziamodai}. We know that $N_{\f g}$  is a $J$-invariant ideal of $\f g$ that contains $\f g'+J \f g'$. In addition, $\f g'_J $ is an ideal of $N_{\f g} $ that satisfies the hypotheses of Proposition \ref{cazziarmati}. This allows us to infer that $\f g'+J\f g'$ acts skew-symmetrically on $\f g_J'$. On the other hand, being $J$ solvable, we know that $\frac{\f g'+J\f g'}{\f g_J'}$ is nilpotent. Now, using the fact that $\f g$ is of real type, we can infer that $\f g'+J\f g' $ is nilpotent and $\f g_J'\subseteq \f z(\f g '+J \f g')$. Finally, applying \cite{AN22}, we conclude that $\f g'+J \f g' $ is $2$-step nilpotent.
\end{proof}

\subsection{The case of higher solvability steps}\label{subsechigher}
In this subsection we collect some results in full generality, dropping the hypothesis of $2$-step solvability. For completely solvable Lie algebras an implication of Proposition \ref{rhocore} can be extended to any solvability step. In order to prove this, we need the following result, which is nothing but an extension of \cite[Proposition 2.1]{LRV15}.

\begin{prop}\label{Prop:cs_sigma}
Let $\f g$ be a solvable Lie algebra endowed with a complex structure $J$. Then
\begin{enumerate}
    \item If $\g$ is completely solvable, then $\rho=0$ if and only if $\sigma=0$.
    \item If $\g$ is of rigid type, then any complex structure on $ \f g$ is Chern--Ricci flat.
\end{enumerate}
\end{prop}
\begin{proof}
Fix an element $X\in \f g$ and consider the adjoint representation
\[
\ad_X-\sqrt{-1}\ad_{JX}=\ad_{X^{1,0}}=\begin{pmatrix}
    A_X  & *\\
    0 & *
\end{pmatrix}\,,
\]
where the matrix is written with respect to the decomposition $\f g_\C=\f g^{1,0}\oplus \f g^{0,1}$. 
By Lie's Theorem we may pick a basis $\{Z_1,\dots,Z_{n}\}$ of $\f g^{1,0}$ with respect to which the block $A_X$ is upper triangular. Set $Z_j=\frac{1}{\sqrt{2}}(e_{2j-1}-\sqrt{-1}e_{2j})$ so that $e_{2j}=Je_{2j-1}$. Choose a Hermitian metric $g$ on $\f g$ having $\{e_1,\dots,e_{2n}\}$ as an orthonormal basis and extend it bilinearly to $\f g_\C$ so that $\{Z_1,\dots,Z_{n}\}$ is a unitary basis. We then compute
\[
\begin{split}
\mathrm{tr}(A_X)&=\sum_{j=1}^{n}g(A_XZ_j,\bar Z_j)=\frac{1}{2}\sum_{j=1}^{n}g\left([X-\sqrt{-1}JX,e_{2j-1}-\sqrt{-1}e_{2j}], e_{2j-1}+\sqrt{-1}e_{2j}\right)\\
&=\frac{1}{2}\left( \mathrm{tr}(\ad_X)-\mathrm{tr}(J\ad_{JX}) - \sqrt{-1}\, \mathrm{tr}(\ad_{JX})-\sqrt{-1}\,\mathrm{tr}(J\ad_X) \right)\,.
\end{split}
\]
Choosing $X\in J\f n$ we have that $\ad_{JX}$ is nilpotent, hence the eigenvalues of $\ad_{X^{1,0}}$ are those of $\ad_X$.

In particular, if $\f g$ is of rigid type, then $\mathrm{tr}(A_X)$ is purely imaginary, and it follows that $\mathrm{tr}(J\ad_{JX})=0$,  for all $X\in J\f n$, because $\f g$ is unimodular. From this we immediately deduce
\[
\rho(X,Y)=-\frac{1}{2}\mathrm{tr}(J\ad_{[X,Y]})+\frac{1}{2}\mathrm{tr}(\ad_{J[X,Y]})=0\,, \qquad X,Y\in \f g\,,
\]
which proves \emph{(2)}.

Assume instead that $\g $ is completely solvable. Then $\mathrm{tr}(A_X)$ is real and so $\mathrm{tr}(J\ad_X)=0$, for all $X\in J\f n $. This implies that $\sigma$ vanishes on $J\f g'$. On the other hand $\rho=d\sigma=-\sigma \vert_{\f g'}$ therefore Chern--Ricci flatness implies that $\sigma$ vanishes on $\f g'+J\f g'$. Now, we have $0=\sigma\vert_{\f g'+J\f g'}=\sigma_{\f g'+J\f g'}$ hence there exists a closed,  nowhere vanishing $(k,0)$-form $\Psi$ on $\g'+J\f g'$, where $k=\dim_\C(\f g'+J\f g')$, see \cite[Theorem 3.2]{AT26}. Pick an orthogonal complement $\f k$ of $\f g'+J\f g'$ in $\f g$ and a basis $z_1,\dots,z_{n-k}$ of $\f k^{1,0}$. Let $\phi^1,\dots,\phi^{n-k}$ be the dual $(1,0)$-forms. Since $\f g' \subseteq \f k^\perp$ the form $ \phi^1\wedge \dots \wedge \phi^{n-k}\wedge \Psi $ is closed and nowhere vanishing on $\f g$ which then implies $\sigma=0$, again by \cite[Theorem 3.2]{AT26}.
\end{proof}

A first consequence of part \emph{(1)} of the above Proposition concerns holomorphic volume forms on completely solvable solvmanifold.

\begin{cor}
Let $M=G/\Gamma$ be a completely solvable solvmanifold equipped with a left-invariant complex structure $J$. Then any holomorphic trivialisation of the canonical bundle of $M$ must be left-invariant.
\end{cor}
\begin{proof}
If the canonical bundle of $M$ is holomorphically trivial clearly $\rho=0$. By Proposition \ref{Prop:cs_sigma} we deduce $\sigma=0$ which implies, together with \cite[Theorem 3.2]{AT26}, that there exists a left-invariant holomorphic volume form on $G$, which then descends to $M$, concluding the proof.
\end{proof}
Another interesting consequence of Corollary \ref{Prop:cs_sigma} is the following.
\begin{cor}\label{Cor:cs_CRF}
Let $\g $ be a unimodular completely solvable Lie algebra equipped with a solvable complex structure $J$. Then $\rho=0$.
\end{cor}

\begin{proof}
We prove the corollary by induction on the solvability step $s$ of $J$. If $s=2$ then the result follows from Proposition \ref{rhocore}. Let us assume that the result is true for $s-1$ and let us prove it for $s$. If $\g$ is $s$-step $J$-solvable the ideal $\g' + J \g'$ is $(s-1)$-step $J$-solvable, thus, by the inductive assumption, $\rho_{\g'+J\f g'}=0$. But then, from Proposition \ref{Prop:cs_sigma} we infer that $\sigma_{\g'+J\f g'}=0$. In particular $\rho=\sigma \vert_{\g'}=0$.
\end{proof}

The purpose of the next example is to show that the converse of Corollary \ref{Cor:cs_CRF} does not hold, unless the solvability step is at most $2$.

\begin{es}
Let $\g$ be the $6$-dimensional Lie algebra with structure equations:
$$
[e_2, e_4]=-e_1\,,\quad [e_3, e_5]=-e_1\,, \quad [e_2,e_6]=-e_2\,,\quad  [e_3, e_6]=-e_3\,, \quad [e_4, e_6]=e_4\,,\quad  [e_5, e_6]=e_5\,.
$$
The equations above define a $3$-step completely solvable Lie algebra which in \cite{FP223} is denoted by $\f s_{6.162}^{1}$.  We endow $\g$ with the following complex structure:
$$
Je_1=e_6\,, \quad Je_2=e_3\,, \quad Je_4=e_5\,.
$$
It turns out that $J$ is simple and $\rho=0$, showing that the converse of Corollary \ref{Cor:cs_CRF} is not true in higher steps of solvability, even assuming complete solvability.
\end{es}

We can extend  the first statement of Theorem \ref{completelysolvSKT1} to higher steps of solvability with a suitable assumption.
\begin{prop}
Let $\f g$ be a unimodular, SKT, completely solvable Lie algebra with solvable complex structure $J$ such that $\f g_J'$ is abelian. Then, $\f g$ is $2$-step nilpotent. 
\end{prop}
\begin{proof}
We will prove the statement using induction on the step of solvability of $J$. Let $J$ be a $s$-step solvable complex structure. If $s=1$, then $\f g'+J\f g' $ is abelian, hence $\f g$ is $2$-step solvable. We can use Theorem  \ref{completelysolvSKT1} to conclude. 
Let us assume that the statement is true for complex structures which are $(s-1)$-step solvable and assume that $J$ is $s$-step solvable. We can then consider $\f k:=\f g'+J\f g'$ which is unimodular, SKT and completely solvable so that $\f k_J'\subseteq \f g_J'$, hence abelian. 
Clearly, the complex structure on $\f k$ is $(s-1)$-step solvable. Hence, by the inductive hypothesis, $\f k$ is $2$-step nilpotent. Now, thanks to \cite[Proposition 3.1]{EFV12} and \cite[Lemma 3.4]{AN22}, we know respectively that $\f z(\f k)$ is a $J$-invariant abelian ideal of $\f g$ and that $\f g_J'\subseteq \f z(\f k)$. Since $J$ is solvable, we can now apply Corollary \ref{Cor:cs_CRF} and infer that $\rho =0$. On the other hand, $\f z(\f k)$ can be chosen as the abelian $J$-invariant ideal of $\f g$ in Proposition \ref{cazziarmati}, which implies that $\f g$ acts skew-symmetrically on $\f z(\f k).$ We are in the position to use complete solvability and infer that $\f g_J'\subseteq \f z(\f k)\subseteq \f z(\f g).$ In order to conclude, we just need to observe that $\f g/ \f z(\f g)$ is a completely solvable Lie algebra with a solvable abelian complex structure. By Corollary \ref{abeliancompletely}, we conclude that $\f g/\f z(\f g)$ is nilpotent and hence $\f g$ is nilpotent as well, as claimed.
\end{proof}
We finally extend to  Theorem \ref{mainfv} under the assumption of complete solvability and a condition on $\f g_J'.$
\begin{prop}
Let $\f g$ be a unimodular, completely solvable Lie algebra endowed with a complex structure $J$ such that $\f g_J'\subseteq {\rm Rad}(\f g)$. If $\f g$ admits a balanced metric, then $J$ is solvable. Furthermore, if an SKT metric also exists, then $\f g$ is K\"ahler.
\end{prop}
\begin{proof}
Assume that $J$ is not solvable. Then $\f h:=\f g/ {\rm Rad}(\f g)$ is a completely solvable Lie algebra with abelian semisimple complex structure $J$. 
Using Proposition \ref{propKilling}, part $(2)$, we have that $-\eta_{\f h}$ is $d$-exact and non-negative, thanks to the fact that $\f g $ is completely solvable. Let $\pi\colon \f g \to \f h$ be the canonical projection onto $\f h$. Thus, $-\pi^*\eta_\h$ is a $d$-exact non-negative $(1,1)$-form, hence it must vanish by Corollary \ref{Michelsohn}. This is equivalent to $\eta_\h=0$, and therefore $B_{\f h}^{1,1}=0$, contradicting the $J$-semisimplicity of $\f h$, by Theorem \ref{Semisimple}. Using Corollary \ref{Cor:cs_CRF}, we conclude that $\rho=0$. The final claim follows from \cite{FV26}.
\end{proof}

\bibliographystyle{siam}
\addcontentsline{toc}{section}{Bibliography}
\bibliography{BiblioLevi}

\end{document}